\documentclass[a4paper,12pt]{amsart}
\usepackage{amsmath,amssymb,amscd,amsfonts}
\usepackage{amsrefs}
\numberwithin{equation}{section}
\newtheorem{theorem}{Theorem}[section]

\newtheorem{rem}[theorem]{Remark}
\newtheorem{defn}[theorem]{Definition}
\newtheorem{lemma}[theorem]{Lemma}

\newtheorem{prop}[theorem]{Proposition}

\newtheorem{corollary}[theorem]{Corollary}

\def\det{\mathop{\rm det}\nolimits}

\def\dbar{\bar\partial}
\def\ddbar{\partial\bar\partial}
\def\d{\partial}

\def\cC{{\mathcal C}}

\def\cD{{\mathcal D}}

\let\ol=\overline

\let\ep=\varepsilon
\let\vp=\varphi

\def\a{\alpha}
\def\z{\zeta}

\def\tgg{\tilde{\Gamma}}
\def\tg{\Gamma}
\def\g{\gamma}
\def\p{\lambda}
\def\q{\nu}
\def\cu{\Upsilon}
\def\cuu{\tilde{\Upsilon}}
\def\a{\alpha}
\def\b{\beta}
\def\cS{{\mathcal S}}

\title[Remark]
{Uniqueness of constant scalar curvature K\"ahler metrics with cone singularities, I: Reductivity}

\author{Long Li}
  \address{Department of Mathematics and Statistics, McMaster University,
1280 Main Street West, Hamilton, ON L8S 4K1, Canada}
  \email[Long Li]{lilong@math.mcmaster.ca}
  
\author{Kai Zheng}
  \address{Mathematics Institute, University of Warwick, Coventry, CV4 7AL, UK}
  \email[Kai Zheng]{K.Zheng@warwick.ac.uk}

\begin{document}

\maketitle

\begin{abstract}
The aim of this paper is to investigate uniqueness of conic constant scalar curvature K\"ahler (cscK) metrics, when the cone angle is less than $\pi$.
We introduce a new H\"older space called $\cC^{4,\a,\b}$ to study the regularities of this fourth order elliptic equation,
and prove that any $\cC^{2,\a,\b}$ conic cscK metric is indeed of class $\cC^{4,\a,\b}$. 
Finally, the reductivity is established by a careful study of the conic Lichnerowicz operator. 
\end{abstract}

\section{Introduction}
Bando-Mabuchi \cite{BM} proved that K\"ahler-Einstein metric on Fano manifolds is unique up to holomorphic automorphisms in 1980s. 
Recently, Berndtsson \cite{Bo} generalized this uniqueness result to conic K\"ahler-Einstein metrics on Fano manifolds, by his observation 
on the convexity of certain energy functional (the $Ding$-functional) along geodesics in the space of all K\"ahler potentials. 
And Berndtsson's method is further generalized to prove the uniqueness of singular K\"ahler-Einstein metrics on \emph{klt-paris} \cite{BBEGZ}.

For the uniqueness of constant scalar curvature K\"ahler \emph{(cscK)} metrics \cite{CPZ}, 
the convexity of the so called Mabuchi's functional turns out to be a very important ingredient. 
And this convexity is proved by the work of Berman-Berndtsson \cite{BB} and also differently by Chen-Li-P\u aun \cite{CLP}. 
Moreover, in a recent work of the first author \cite{Li}, the convexity is generalized to conic Mabuchi's functional along any geodesics in the space $\cC^{1,\bar 1}_{\b}$. 
However, according to the work of Calamai-Zheng \cite{CZ}, the existence of such geodesic is only known for small angles ($\beta < 1/2$) up to this stage. 
Therefore, we will first discuss the uniqueness of conic cscK metrics when the cone angle is less than $\pi$ in this paper.  

\begin{defn}
A pair ($\a,\b$) of real numbers is named to be \emph{under the small angle conidtion} if $\a\in (0,1)$ and $\b\in(0,1/2)$ satisfying $\a\b < 1- 2\b$.
\end{defn}

Unlike K\"ahler-Einstein case, there are several different definitions of singular cscK metrics with cone like singularities along a divisor (\cite{Chen}, \cite{Li}, \cite{Has}). 
In the work \cite{Li}, it provided two ways to investigate this problem. 
On the one hand, we can place certain curvature conditions to the geometric equation $(R_{\varphi} = \underline{R})$. 
On the other hand, we can input enough regularities to the Monge-Amp\`ere equation associated to the geometric equation. 
And we will adapt to the latter in this paper. 

Donaldson \cite{Don12} introduced several H\"older spaces ($\cC^{,\a,\b}, \cC^{2,\a,\b}$) to study the regularities of the Laplace operator associated with a conic metric. 
Note that the Laplace operator is exactly the principle part of the linearized operator for K\"ahler-Einstein equation. 
Similarly, the linearized operator for cscK metrics is called the Lichnerowicz operator $\cD$, which is a 4th. order elliptic differential operator. 
In order to study the regularity property of this operator, we introduced a new H\"older space $\cC^{4,\a,\b}$, which is simply defined by one more relation to $u\in\cC^{2,\a,\b}$
\begin{equation}
\label{000}
\Delta_{\Omega_0} u \in \cC^{2,\a,\b}, 
\end{equation}
where $\Omega_0$ is a local model cone metric (Equation \ref{new3}). 

This is a well defined, Banach space equipped with a norm $||\cdot ||_{\cC^{4,\a,\b}}$. 
Very interestingly, we can estimate the growth rates for all 3rd. and 4th. order derivatives of any $u\in\cC^{4,\a,\b}$ from this simple relation (equation (\ref{000})). 
The computation for the 3rd. order derivatives are very similar to Brendle's calculation \cite{Br},
but it is more subtle for the 4th. order derivatives (Riemannian curvatures). And we proved the following regularity theorem. 

\begin{theorem}[Theorem \ref{thm-reg}]
Suppose $u\in\cC^{2,\a,\b}$ is the potential of a conic cscK metric $\omega$. Then $u$ is of class $\cC^{4,\a,\b}$ under the small angle condition. 
\end{theorem}

This is basically done by Donaldson's estimates and Brendle's trick \cite{Br}. 
Thanks to these two important tools, we also proved a rough Schauder estimate for $\cD$ operator induced by a conic cscK metric. 
And this gives a chance to lift the regularities of potentials which are in the kernel of $\cD$ operator. 
Finally, we generalized the Lichnerowicz-Calabi theorem to $\cC^{2,\a,\b}$ conic cscK metrics (Definition \ref{def-new3}). 

Let $\mathfrak{h}(X;D)$ be the space of all holomorphic vector fields tangential to the divisor $D$, and there is a one-one correspondence between 
$\mathfrak{h}(X;D)$ and the automorphism group $Aut(X;D)$ (all automorphism of $X$ fixing $D$). 

\begin{theorem} [Theorem \ref{thm-reductive}]
\label{thm-main}
Suppose $\omega$ is a $\cC^{2,\a,\b}$ conic cscK metric on $X$ under the small angle condition. 
Then the Lie algebra $\mathfrak{h} (X;D)$ has a semidirect sum decomposition:
\begin{equation}
\mathfrak{h}(X;D) = \mathfrak{a}(X;D) \oplus \mathfrak{h'} (X;D),
\end{equation}
where $\mathfrak{a}(X;D)$ is the complex Lie subalgebra of $\mathfrak{h}(X;D)$ consisting of all parallel holomorphic vector fields tangential to $D$,
and $\mathfrak{h'}(X;D)$ is an ideal of $\mathfrak{h}(X;D)$ consisting of the image under $grad_g$ of the kernel of $\cD$ operator.

Furthermore $\mathfrak{h}'(X;D)$ is the complexification of a Lie algebra consisting of Killing vector fields of $X$ tangential to $D$. 
In particular $\mathfrak{h}'(X;D)$ is reductive. 
\end{theorem}

The theorem is proved by the following decomposition
\begin{equation}
\label{0000}
Y = Y^{//} \oplus Y^{\perp}
\end{equation}
where $Y$ is any holomorphic vector field on $X$ tangential to the divisor $D$,
$Y^{//}$ is a parallel holomorphic vector field tangential to $D$, 
and $Y^{\perp}$ is the image under $grad_g$ of some function $u\in \ker\cD$. 
This is done by proving a kind of Hodge decomposition to the lifting of $Y$ first, 
and then an integration by parts formula for the conic Lichnerowicz operator establishes the one-one correspondence. 

Moreover, notice that the conic Lichnerowicz operator $\cD$ is still a real operator for conic cscK metrics. And this enable us to conclude the following reductivity. 

\begin{corollary}
\label{thm-main1}
Suppose a compact complex K\"ahler manifold $X$ admits a $\cC^{2,\a,\b}$ conic cscK metrics under the small angle condition. 
Then the automorphism group $Aut(X,D)$ is reductive. 
\end{corollary}

In the sequel paper \cite{LZ2}, we establish a linear theory for the Lichnerowicz operator, 
which implies that the operator $\cD: \cC^{4,\a,\b}\rightarrow \cC^{,\a,\b}$ has closed range.
And this enable us to generalize the Bifurcation technique \cite{CPZ} to conic cscK metrics. 

$\mathbf{Acknowledgement}$: We are very grateful to Prof. Xiuxiong Chen, for his continuous encouragement in mathematics, 
and we would like to thank Prof. S. Donaldson who suggested this problem to us. 
We also want to thank Dr. Chengjian Yao, Dr. Yu Zeng, and Dr. Yuanqi Wang for lots of useful discussions. 

The first author wants to thank Prof. I. Hambleton and Prof. M. Wang for many supports. 
The second author was supported by the EPSRC on a Programme Grant entitled "Singularities of Geometric Partial Differential Equations" reference number EP/K00865X/1.

\section{ A new H\"older space}
Let $(X, \omega_0)$ be a compact K\"ahler manifold with complex dimension $n$, and $D$ is a simple smooth divisor on $X$.
Suppose $(L_D, h)$ is the line bundle induced by the divisor $D$ with a metric $h$,
and $s$ is a non-trivial holomorphic section of it.

For arbitrary point $p\in D$, we can introduce a local coordinate chart $(\zeta, z_2,\cdots, z_n)$ in an open neighborhood $U$ centered at $p$, 
such that the divisor is defined by the zero locus of $\zeta$ in $U$. And we call it as the \emph{z-coordinate}.
Writing $\zeta = \rho e^{i\theta}$ in polar coordinate, we can introduce another 
local coordinate as $(\xi, w_2,\cdots, w_n) $, where 
$$\xi: = r e^{i\theta}, \ \ \ r = \rho^{\beta},$$
and $w_j:= z_j$ for all $j>1$. And we call it as the \emph{w-coordinate} (Not holomorphic!). 

In $z$-coordinate, we denote $\Omega_0$ by the \emph{local model cone metric} 
$$\Omega_0 = \b^{2}|\z|^{2\b-2} id\z\wedge d\bar\z + \sum_{j>1} idz_j \wedge d\bar z_j.$$
Then a conic K\"ahler metric $\omega$ is defined as a positive closed $(1,1)$ current on X, smooth on $X-D$, such that 
it is quasi-isometric to $\Omega_0$ near each point $p$ on $D$. 
That is to say, on a open neighborhood $U$ of any $p\in D$, we have 
$$ C^{-1} \Omega_0 \leq \omega \leq C\Omega_0, $$
for some constant $C>0$.

\subsection{Higher order H\"older spaces}
From now on, we adapt to the following conventions:

\begin{enumerate}
\smallskip
\item[(a)] Greek letters like $\a, \b, \mu, \nu,\cdots $ represent all indexes $\z, 2,\cdots, n$;

\item[(b)] English letters like $k, l, i, j, \cdots$ only reprent indexes $2,\cdots, n$;

\item[(c)] $\d$ and $\dbar$ denotes the ordinary partial differential operator;

\item[(d)] The operator $\nabla $ denotes the Chern connection with respect to certain K\" ahler metric,
and we also use $u_{,\a\bar\b\g}$ to denote the acting of the connection to the function $u$ in $\a, \bar\b, \g$ directions

\end{enumerate}

According to Donaldson \cite{Don12}, we denote $\mathcal{C}^{,\a,\b}$ by the space of all real valued functions $f$ of the form $f(\zeta, z_2,\cdots, z_n) = \tilde{f}(\xi, w_2,\cdots, w_n)$ where $\tilde{f}\in C^{\a}$.  And $\mathcal{C}^{,\a,\b}_0$ denotes the space of all functions $f\in \mathcal{C}^{,\a,\b}$ such that $f(0, z_2,\cdots, z_n) =0$. 
A $(1,0)$ form 
$$\tau: = \tau_{\z} d\z + \sum_{k>1} \tau_k dz_k $$
 in $z$-coordinate is said to be of class $\mathcal{C}^{,\a,\b}$ if its coefficients satisfy 
$$\rho^{1-\beta}\tau_{\z} \in \cC^{,\a,\b}_0; \ \ \ \tau_k \in \cC^{,\a,\b}.$$
And a $(0,1)$ form $\tau'$ is of class $\cC^{,\a,\b}$ if $\ol{\tau'}\in \cC^{,\a,\b}$. 
Moreover, a $(1,1)$- form $$\sigma: = \sigma_{\z\bar{\z}} d\z\wedge d\bar{\z} + \sum_{k>1} \sigma_{\z\bar{l}} d\z\wedge d\bar{z}_l 
+ \sum_{k>1} \sigma_{k\bar{\z}} dz_k\wedge d\bar{\z} + \sum_{k, l>1} \sigma_{k\bar{l}} dz_k\wedge d\bar{z}_l$$
is said to be of class $\cC^{,\a,\b}$ if  
$$ \sigma_{k\bar{l}} \in \cC^{,\a,\b};\  \rho^{1-\b}\sigma_{\z\bar{l}}, \rho^{1-\b} \sigma_{k\bar{\z}} \in \cC^{,\a,\b}_0; \  \rho^{2-2\b} \sigma_{\z\bar{\z}}\in \cC^{,\a,\b}.$$
Donaldson \cite{Don12} also introduced the following H\"older space of the second order as 
$$ \cC^{2,\a,\b}: = \{ u\in C^{2}_{\mathbb{R}}(X-D)|\ \  u, \d u, \ddbar u \in \cC^{,\a,\b} \}.$$  
We call a complex valued function $f$ is of class $\cC^{,\a,\b}$ $(\cC^{2,\a,\b})$, 
if both the real part and imaginary part of $f$ are of class $\cC^{,\a,\b}$ ($\cC^{2,\a,\b}$).

Moreover, we can consider a new H\"older space of the 4th. order as 
\begin{defn}
$$ \cC^{4,\a,\b}: = \{ u\in C^{2,\a,\b}\cap C_{\mathbb{R}}^4(X-D) |\ \ \rho^{2-2\b} \d_\z \d_{\bar\z} u\in\cC^{2,\a,\b},\  \d_k \d_{\bar l} u\in \cC^{2,\a,\b} \}.$$
\end{defn}
Notice that the operator $i\ddbar$ is real, and the definition makes sense even when $u$ is complex valued.
\begin{rem}
According to Brendle's computation (Appendix A.1 \cite{Br}), the small angle condition gives us a slightly better growth estimate for $\cC^{2,\a,\b}$ functions.
In fact, if $u\in \cC^{2,\a,\b}$, then we have 
$$  \rho^{1-2\b}(\d u)_{\z} \in \cC^{,\a,\b}.  $$
\end{rem}

Recall that in $z$-coordinate, we denote $\Omega_0$ by the local model cone metric 
\begin{equation}
\label{new3}
\Omega_0 = \b^{2}|\z|^{2\b-2} id\z\wedge d\bar\z + \sum_{j>1} idz_j \wedge d\bar z_j.
\end{equation}
Thanks to Donaldson's estimate, we have another description of the H\"older space $\cC^{4,\a,\b}$.
\begin{prop}
We have the following equivalence 
\label{prop-C4}
$$ \cC^{4,\a,\b} = \{ u\in\cC^{2,\a,\b}\cap C_{\mathbb{R}}^4(X-D) |\ \ \Delta_{\Omega_0} u\in \cC^{2,\a,\b}\}.$$
\end{prop}
\begin{proof}
Suppose $u$ is a $\cC^{2,\a,\b}$ function such that $\Delta_{\Omega_0}u \in \cC^{2,\a,\b}$, and then we have
\begin{eqnarray}
\Delta_{\Omega_0} (\d_k\d_{\bar l}u) &=& \b^{-2}|\z|^{2-2\b} \d_\z \d_{\bar \z} \d_{k}\d_{\bar l} u + \sum_{i>1} \d_{i}\d_{\bar i}\d_k\d_{\bar l}u 
\nonumber\\
&=& \d_{k}\d_{\bar l} \Delta_{\Omega_0}u \in \cC^{,\a,\b}.
\end{eqnarray}
Then $\d_k\d_{\bar l}u$ is of class $\cC^{2,\a,\b}$ by Donaldson's estimates. 
This implies that 
$$ |\z|^{2-2\b}\d_{\z}\d_{\bar\z}\d_k\d_{\bar l}u  \in \cC^{,\a,\b}, $$
and we also have
\begin{eqnarray}
\label{300}
\Delta_{\Omega_0} (|\z|^{2-2\b}\d_{\z}\d_{\bar\z} u)
 &=& \b^{-2} |\z|^{2-2\b} \d_{\z}\d_{\bar\z} ( |\z|^{2-2\b}\d_\z\d_{\bar\z}u)
+ \sum_{k>1} |\z|^{2-2\b}\d_{\z}\d_{\bar\z}\d_k\d_{\bar k}u
\nonumber\\
&=& |\z|^{2-2\b} \d_{\z}\d_{\bar\z}\Delta_{\Omega_0} u.
\end{eqnarray}
Hence $|\z|^{2-2\b}\d_\z\d_{\bar\z} u\in \cC^{2,\a,\b}$,
and the result follows. 
\end{proof}

From this new description, we infer that the definition of $\cC^{4,\a,\b}$ is independent of holomorphic coordinate charts, and it is a Banach space.
And we will discuss further details later (see Section \ref{sec-fund}, \emph {Fundamentals of $\cC^{4,\a,\b}$ space}). 

Now we can estimate the growth rate of the 3rd. order derivatives of  a $\cC^{4,\a,\b}$ function by Donaldson's estimate. 
\begin{lemma}
\label{lem-3rd}
For any $u\in \cC^{4,\a,\b}$, we have 
$$ \d_{k}\d_{\bar{l}}\d_{i} u \in \cC^{,\a,\b},$$
$$ \rho^{1-\b} \d_k\d_{\bar{\z}}\d_{i} u \in \cC_0^{,\a,\b},$$
$$ \rho^{1-\b} \d_{\z}\d_{\bar{l}}\d_k u \in\cC_0^{,\a,\b},  $$
$$ \rho^{2-2\b} \d_{\z}\d_{\bar{\z}}\d_k u \in \cC^{,\a,\b}, $$
for all $k,l, i >1$. 
\end{lemma}
\begin{proof}
According to Donaldson's estimate, it is enough to prove on $X-D$
$$ \Delta_{\Omega_0} v \in \cC^{,\a,\b},$$
where $v = \d_k u \in \cC^{,\a,\b}$, and $\Omega_0$ is the local model metric near a point on $D$. 
Then we can compute the Laplacian as 
$$ \Delta_{\Omega_0} (\d_k u) = \b^{-2} \rho^{2-2\b} \d_{\z}\d_{\bar{\z}} \d_k u + \sum_{j>1}\d_j\d_{\bar j}\d_k u ,$$
but $\d (\rho^{2-2\b}\d_{\z}\d_{\bar{\z}}u)$ is of class $\cC^{,\a,\b}$ as a $(1,0)$ form
$$ \d_k(|\z|^{2-2\b} \d_{\z}\d_{\bar{\z}} u) = |\z|^{2-2\b} \d_{\z}\d_{\bar{\z}}\d_k u \in \cC^{,\a,\b}, $$
and the lemma follows since $\d_j\d_{\bar j}\d_k u \in \cC^{,\a,\b}$ for each $j>1$. 
\end{proof}

\begin{lemma}
\label{lem-3rd-1}
We have 
$$ \rho^{2-2\b} \left( \d_{\z}\d_{\bar{l}}\d_{\z} u + \frac{1-\b}{\z} \d_{\z}\d_{\bar{l}} u \right) = O(\rho^{\a\b}), $$
$$ \rho^{3-3\b} \left (\d_{\z}\d_{\bar{\z}}\d_{\z} u + \frac{1-\b}{\z} \d_{\z}\d_{\bar{\z}} u \right) \in \cC_0^{,\a,\b}. $$
\end{lemma}
\begin{proof}
The first estimate is true thanks to Lemma (\ref{lem-3rd}) and Brendle's appendix \cite{Br}. 
The second estimate follows from the computation 
\begin{eqnarray}
\label{110}
|\z|^{1-\b}\d_{\z} (|\z|^{2-2\b} \d_{\z}\d_{\bar{\z}} u) &=& |\z|^{1-\b}\{ (1-\b)\bar{\z}^{1-\b} \z^{-\b} \d_{\z}\d_{\bar{\z}}u  + |\z|^{2-2\b}\d_{\z}\d_{\bar{\z}} \d_{\z} u \}
\nonumber\\
&=& |\z|^{3-3\b} \left( \d_{\z}\d_{\bar{\z}}\d_{\z} u + \frac{1-\b}{\z}\d_{\z}\d_{\bar{\z}}u \right) \in \cC^{,\a,\b}_0.
\end{eqnarray}
\end{proof}

Suppose $s$ is a holomorphic section of $D$, and $h$ is a Hermitian metric on the line bundle $D$. Put 
$$\Omega: = \omega_0 + \lambda dd^c |s|^{2\b}_h. $$
For $\lambda >0$ small enough, $\Omega$ is a conic K\"ahler metric, and we call it as the \emph{model conic K\"ahler metric } with angle $2\pi\b$ along $D$. 
Notice that the potential $|s|^{2\beta}_h$ is in the H\"older space $\cC^{2,\a,\b}$. 
Moreover, the growth rate of the 3rd. and 4th. order derivatives of $|s|^{2\b}_h$ is computable.

\begin{defn}
\label{def-new3}
We call $\omega_{\vp}: = \omega_0 + dd^c \vp$ as a \emph{$\cC^{2,\a,\b} (\cC^{4,\a,\b})$ conic K\"ahler metric} if 
$\omega_{\vp}$ is a conic K\"ahler metric and $\vp$ is of class $\cC^{2,\a,\b} (\cC^{4,\a,\b})$. 
\end{defn}
Notice that any smooth function is of class $\cC^{4,\a,\b}$ under the small angle condition. 
\begin{lemma}
\label{lem-model}
Under the small angle condition, the model metric $\Omega$ is a $\cC^{4,\a,\b}$ conic K\"ahler metric.
\end{lemma}
\begin{proof}
Writing $ \omega_0  = dd^c \phi_0$ locally, it is enough to prove that the function $\phi_0 + |s|^{2\b}_h$ is of class $\cC^{4,\a,\b}$. 
But this directly follows from Brendle's computation (Lemma 3.3, \cite{Br}). 
\end{proof}

\begin{rem}
Note here mixed derivatives are not involved in the calculation of Lemma (\ref{lem-model}). 
In fact, the derivative $ \rho^{1-\b} \d_\z\d_{\bar l} (\phi_0 + |s|^2_h)$ is NOT of class $\cC^{2,\a,\b}$ by Brendle's computation. 
\end{rem}
Now we want to compare a general $\cC^{2,\a,\b}$ conic metric with the model metric $\Omega$. 
\begin{defn}
Suppose a pair $(\a,\b)$ is under the small angle condition, and let $\omega_{\vp}$ be a $\cC^{2,\a,\b}$ conic K\"ahler metric. 
We say $\omega_{\vp}$ has the model growth in the 3rd. order if its Christoffel symbols satisfy 
\begin{eqnarray}
\Gamma_{ik}^j \in \cC^{,\a,\b}, \\
\rho^{1-\b}\Gamma^k_{i\z} \in \cC^{,\a,\b}_0,\ \ \rho^{\b-1} \Gamma^{\z}_{ik} \in \cC^{\a,\b}, \\
\Gamma_{i\z}^{\z}\in \cC^{,\a,\b}, \ \ \rho^{2-2\b}\Gamma^k_{\z\z}\in\cC^{,\a,\b}_0,\\
\rho^{1-\b}(\Gamma^{\z}_{\z\z} + (1-\b)\z^{-1}) \in \cC^{,\a,\b}_0. 
\end{eqnarray}
\end{defn}

\begin{prop}
\label{prop-chris}
A $\cC^{4,\a,\b}$ conic K\"ahler metric has the model growth in the 3rd. order.
\end{prop}
\begin{proof}
We have the following growth estimates for the metric thanks to Lemma (\ref{lem-3rd}) and (\ref{lem-3rd-1})
\begin{eqnarray}
\d_i g_{k\bar{l}} \in \cC^{,\a,\b}, \\
\rho^{1-\beta} \d_i g_{\z\bar{l}} \in \cC^{,\a,\b}_0,\\
\rho^{2-2\b} \d_i g_{\z\bar{\z}} \in \cC^{,\a,\b}, \\
\rho^{2-2\b} \left[ \d_{\z} g_{\z\bar{l}} + \frac{1-\beta}{\z} g_{\z\bar{l}} \right] \in \cC^{,\a,\b}_0, \\
\rho^{3-3\b} \left[ \d_{\z}g_{\z\bar{\z}} + \frac{1-\b}{\z}g_{\z\bar{\z}} \right] \in \cC^{,\a,\b}_0.
\end{eqnarray}
Then the result follows by a standard computation (Lemma 6.3, \cite{Br}). 
\end{proof}

\begin{prop}
\label{prop-conn}
Let $\nabla$ be the covariant derivative of a $\cC^{4,\a,\b}$ conic K\"ahler metric.
Then for any $u\in \cC^{4,\a,\b}$, we have 
$$ \nabla_{k}\nabla_{\bar{l}}\nabla_{i} u\in \cC^{,\a,\b}, $$
$$ \rho^{1-\b} \nabla_k\nabla_{\bar{\z}} \nabla_i u \in \cC^{,\a,\b}_0, $$
$$ \rho^{1-\b} \nabla_{\z}\nabla_{\bar{l}}\nabla_{k} u \in \cC^{,\a,\b}_0,$$
$$ \rho^{2-2\b} \nabla_{\z}\nabla_{\bar{\z}}\nabla_{k} u\in \cC^{,\a,\b},$$
$$ \rho^{2-2\b} \nabla_{\z}\nabla_{\bar{l}} \nabla_{\z} u \in \cC^{,\a,\b}_0. $$
for all $i,k,l>1$.  
\end{prop} 
\begin{proof}
Thanks to Lemma (\ref{lem-3rd}), (\ref{lem-3rd-1}) and Proposition (\ref{prop-chris}), we run the same computation with Brendle's (Proposition 6.4, \cite{Br}). 
\end{proof}

\begin{lemma}
We have 
$$ \rho^{3-3\b} \nabla_\z\nabla_{\bar\z}\nabla_\z u \in\cC^{,\a,\b}_0. $$
\end{lemma}
\begin{proof}
Compute the third order covariant derivative in normal direction only 
\begin{eqnarray}
|\z|^{3-3\b} \nabla_{\z} u_{,\z\bar{\z}} &=& |\z|^{3-3\b} ( \d_{\z}u_{,\z\bar{\z}} - \Gamma^{\mu}_{\z\z}u_{,\mu \bar{\z}})
\nonumber\\
&=& |\z|^{3-3\b} \left( \d_{\z}\d_{\bar\z} \d_{\z}u + \frac{1-\b}{\z} \d_{\z}\d_{\bar\z} u \right)
\nonumber\\
&-& |\z|^{3-3\b} \sum_{k>1}\Gamma^{k}_{\z\z} u_{k\bar\z}
\nonumber\\
&-& |\z|^{3-3\b} (\Gamma^{\z}_{\z\z} + (1-\b) \z^{-1}) \d_{\z}\d_{\bar\z}u.
\end{eqnarray}
Since 
$$  |\z|^{3-3\b} \left( \d_{\z}\d_{\bar\z} \d_{\z}u + \frac{1-\b}{\z} \d_{\z}\d_{\bar\z} u \right) \in \cC^{,\a,\b}_0,$$
and 
$$ |\z|^{1-\b} (\Gamma^{\z}_{\z\z} + (1-\b) \z^{-1}) \in \cC^{,\a,\b}_0.$$
The result follows. 
\end{proof}

\subsection{Curvature tensors}
Before continuing the calculation, we would like to introduce some conventions
\begin{defn}
\label{def-tensor}
A tensor in the $z$-coordinate denoted by
$$\Xi^{\z,\cdots,\bar\z,A}_{\z,\cdots,\bar\z,B}$$ 
is called cone admissible if 
$$ |\z|^{(p-q)(1-\b)}\Xi^{\z,\cdots,\bar\z,A}_{\z,\cdots,\bar\z,B} \in \cC^{,\a,\b}, $$
where $p$ is number of all $\z,\bar\z$ appearing in the upper indices of $\Xi$, $q$ is the number of all $\z,\bar\z$ appearing in the lower indices,
and $A,B$ only consist of indices which are bigger than $1$. 
\end{defn}

\begin{defn}
\label{def-christ}
Let $\Gamma^{\gamma}_{\mu\lambda}$ be the Christoffel symbols of a conic K\"ahler metric.
We denote by $\tilde{\Gamma}^{\gamma}_{\mu\lambda}$ the following tensor
$$ \tilde{\Gamma}^{\z}_{\z\z} = \Gamma^{\z}_{\z\z} + (1-\b)\z^{-1},$$
and $$\tilde{\Gamma}^{\gamma}_{\mu\lambda} = {\Gamma}^{\gamma}_{\mu\lambda}, $$
if any one of the indices $\gamma, \mu, \lambda$ is not equal to $\z$.
\end{defn}

\begin{rem}
\label{rem-chirs}
If $\omega$ is under the model growth in the 3rd. order, then it is easy to see that its Christoffel symbols 
$\Gamma^{\gamma}_{\mu\lambda}$ are all cone admissible tensors except one $\Gamma^{\z}_{\z\z}$, but
the tensor $\tilde\Gamma^{\gamma}_{\mu\lambda}$ are cone admissible tensor for all indices $\gamma, \mu, \lambda$.
\end{rem}

Moreover, take a local function in $z$-coordinate 
$$\phi: = -(1-\b)\log |\z|^2.$$
We introduce a new first order differential operator $\d^{\phi}$ as 
$$ \d^{\phi}: = |\z|^{2\b-2}\d ( |\z|^{2-2\b} \cdot ), $$
and in each direction, it can be computed as 
$$ \d^{\phi}_{\z} = \d_{\z} +  (1-\b)\z^{-1}\cdot,\ \ \d^{\phi}_{k} = \d_k.$$
If we denote by $\d^{\phi}_{\bar\mu}$ the direction component of $\ol{\d^{\phi}}$, then an easy calculation shows the
commutative relation
$$ \d^{\phi}\dbar + \dbar\d^{\phi} = 0,$$ on $X-D$.
Furthermore, we can re-arrange our connection operator when it acts on a $1$ form $\tau$
\begin{eqnarray}
\label{newcon}
\nabla_{\z}\tau_{\z} &=& \d_{\z} \tau_{\z} - \Gamma^{\gamma}_{\z\z}\tau_{\gamma}
\nonumber\\
&=& (\d_{\z}\tau_{\z} + (1-\b)\z^{-1} \tau_\z) - (\Gamma^{\z}_{\z\z} + (1-\b)\z^{-1})\tau_{\z} - \sum_{k>1}\Gamma^{k}_{\z\z}\tau_k
\nonumber\\
&=& \d^{\phi}_{\z} \tau_\z - \tilde{\Gamma}^{\gamma}_{\z\z}\tau_\gamma.
\end{eqnarray}
and for $k>1$
\begin{eqnarray}
\label{newcon2}
\nabla_{k}\tau_{\z} &=& \d_k \tau_\z - \Gamma^{\gamma}_{k\z}\tau_\gamma
\nonumber\\
&=& \d^{\phi}_{k}\tau_{\z} - \tgg^{\gamma}_{k\z} \tau_\g
\end{eqnarray}
In order to emphasis this new decomposition of $\nabla$, we introduce the following convention $\nabla^{\phi}$
\begin{equation}
\label{1111}
 \nabla^{\phi}_{\z}\tau_{\z} =  \d^{\phi}_\z \tau_\z - \tilde{\Gamma}^{\gamma}_{\z\z}\tau_\gamma,
\end{equation}
and 
\begin{equation}
\label{1112}
 \nabla^{\phi}_{\mu}\tau_{\lambda} = \nabla_{\mu}\tau_{\lambda},   
\end{equation}
if any one of $\mu, \lambda$ is not $\z$. 

Now we can discuss the 4th. order derivatives of a function $u\in\cC^{4,\a,\b}$. 
\begin{lemma}
\label{lem-4th}
We have 
$$ \rho^{1-\b}\d_i\d_{\bar j}\d_\z\d_{\bar l} u \in \cC^{,\a,\b}_0,$$
$$ \rho^{3-3\b} \left( \d_k\d_\z\d_{\bar\z}\d_\z u + \frac{(1-\b)}{\z} \d_k\d_{\bar\z}\d_\z u \right) \in \cC^{,\a,\b}_0,$$
$$ \rho^{3-3\b} \left( \d_{\bar{l}}\d_{\bar\z}\d_\z\d_{\bar\z} u + \frac{(1-\b)}{\bar\z} \d_{\bar l}\d_{\bar\z}\d_\z u \right) \in \cC^{,\a,\b}_0, $$
for all $i,j,k,l>1$.
\end{lemma}
\begin{proof}
Recall that $|\z|^{2-2\b}\d_\z\d_{\bar\z}u$ and $\d_i\d_{\bar j} u$ are both of class $\cC^{2,\a,\b}$, and we can take two derivatives of it 
\begin{eqnarray}
\label{220}
|\z|^{1-\b}\d_k\d_{\bar\z} (|\z|^{2-2\b} \d_{\z}\d_{\bar\z} u) &=& |\z|^{1-\b}\d_k ( |\z|^{2-2\b} \d_{\bar\z}\d_\z\d_{\bar\z} u + (1-\b) \z^{1-\b}\bar\z^{-\b} \d_\z\d_{\bar\z} u)
\nonumber\\
&=&  |\z|^{3-3\b} \left( \d_k\d_\z\d_{\bar\z}\d_\z u + \frac{(1-\b)}{\z} \d_k\d_{\bar\z}\d_\z u \right) \in \cC^{,\a,\b}_0,
\end{eqnarray}
\end{proof}

\begin{lemma}
\label{lem-3partial}
We have 
\begin{equation}
\label{ap-005}
\z\cdot \d_\z\d_{\bar l}u \in \cC^{2,\a,\b},
\end{equation}
\begin{equation}
\label{ap-3par-tag}
|\z|^{2-2\b}\d_k \d_{\bar l} u \in \cC^{2,\a,\b}
\end{equation}
and 
\begin{equation}
\label{ap-006} 
|\z|^{2-2\b} \d_\z\d_{\bar l}u \in \cC^{2,\a,\b}, 
\end{equation}
for all $l >1$. 
\end{lemma}
\begin{proof}
On the tangential direction $z_j$, it is easy to have 
$$ \d_j\d_{\bar j} (\z \d_\z\d_{\bar l} u) = \z (\d_\z \d_{\bar l} \d_j\d_{\bar j} u) = O(\rho^{\b}). $$
On the normal direction, we have 
$$ \d_{\z}\d_{\bar\z} (\z\d_{\z}\d_{\bar l} u) = \z \d_{\z}\d_{\bar\z}\d_\z\d_{\bar l} u + \d_{\bar\z}\d_{\z}\d_{\bar l}u  $$
But thanks to Lemma (\ref{lem-4th}), we have 
$$  \z |\z|^{2-2\b} \d_{\z}\d_{\bar\z}\d_\z\d_{\bar l} u \in\cC^{,\a,\b}, \ \  |\z|^{2-2\b}\d_{\bar\z}\d_{\z}\d_{\bar l}u \in\cC^{,\a,\b}, $$
Then equation (\ref{lem-3partial}) follows. Now equation (\ref{ap-3par-tag}) follows from the calculation 
\begin{eqnarray}
\label{ap-eqq}
|\z|^{2-2\b} \d_\z\d_{\bar\z} ( |\z|^{2-2\b}\d_k \d_{\bar l} u) &=& |\z|^{4-4\b} \{ \d_\z\d_{\bar\z}u_{, k\bar l} + (1-\b)\z^{-1}\d_{\bar\z} u_{,k\bar l}
\nonumber\\
& + & (1-\b) \bar{\z}^{-1} \d_\z u_{, k\bar l} + (1-\b)^2 |\z|^{-2} u_{,k\bar l} \}
\nonumber\\
&=& O (\rho^{2-4\b}).
\end{eqnarray}
Moreover, we have 
\begin{eqnarray}
 |\z|^{2-2\b} \d_\z\d_{\bar\z} ( |\z|^{2-2\b}\d_\z \d_{\bar l} u) & = & |\z|^{4-4\b} \{ \d_{\bar l}\d^{\phi}_\z \d_{\bar\z}\d_\z u  
\nonumber\\
& + & (1-\b) \bar{\z}^{-1} (\d_\z\d_\z\d_{\bar l}u + (1-\b)\z^{-1} \d_\z\d_{\bar l}u ) \} 
\nonumber\\
&=& O( \rho^{ 1-2\b}).
\end{eqnarray}
The last line follows from Lemma (\ref{lem-3rd-1}), and equation (\ref{ap-006}) follows. 
\end{proof}

\begin{lemma}
\label{lem-vol}
Let $u$ and $v$ be two $\cC^{4,\a,\b}$ functions, and then we have  
$$ |\z|^{2-2\b} (\d_{\z}\d_{\bar l} u)( \d_k\d_{\bar\z}v) $$
is of class $\cC^{2,\a,\b}$ for any fixed $k,l >1$. In particular, 
$$ |\z|^{2-2\b} (\d_{\z}\d_{\bar l} u)( \d_k\d_{\bar\z}u) \in\cC^{2,\a,\b},$$
for any $k,l>1$.
\end{lemma}
\begin{proof}
Put 
$$ \Phi: = |\z|^{2-2\b} (\d_{\z}\d_{\bar l} u)( \d_k\d_{\bar\z}v),$$
and then it is enough to prove $\Delta_{\Omega_0} \Phi $ is of class $\cC^{,\a,\b}$. But on the tangential directions, we have 
\begin{eqnarray}
\label{tag}
\d_j\d_{\bar j} \Phi & = & |\z|^{2-2\b} \{ ( \d_j\d_{\bar j}\d_\z\d_{\bar l}u\cdot \d_k\d_{\bar\z} v + \d_j\d_{\bar{l}}\d_{\z} u\cdot \d_{\bar j}\d_k \d_{\bar\z}v) 
\nonumber\\
& + &  (\d_{\bar j}\d_\z \d_{\bar l}u \cdot \d_j\d_{\bar\z}\d_k v + \d_j\d_{\bar j}\d_k\d_{\bar\z} v\cdot \d_\z\d_{\bar l} u ) \}.
\end{eqnarray}

On the normal direction to the divisor, we have 
\begin{eqnarray}
\label{normal}
|\z|^{2-2\b} \d_{\z}\d_{\bar\z} \Phi &=& |\z|^{4-4\b} \{ (\d_{\bar l}\d_\z\d_{\bar\z} \d_\z u + (1-\b)\z^{-1} \d_{\bar l} \d_{\z} \d_{\bar\z} u)\cdot \d_k\d_{\bar\z}v
\nonumber\\
&+ &  (\d_{k}\d_{\bar\z}\d_{\z} \d_{\bar\z} v + (1-\b) \bar{\z}^{-1} \d_{k} \d_{\bar\z} \d_{\z} v) \cdot \d_\z\d_{\bar l}u
\nonumber\\
&+ & ( \d_{\z}\d_{\bar l}\d_{\z} u + (1-\b)\z^{-1}\d_{\bar l}\d_{\z} u )\cdot ( \d_{\bar\z}\d_{k}\d_{\bar\z} v + (1-\b)\bar\z^{-1}\d_{k}\d_{\bar\z} v ) 
\nonumber\\
&+&  \d_{\bar\z}\d_{\z}\d_{\bar l} u \cdot \d_{\z}\d_{\bar\z}\d_{k} v\}
\end{eqnarray}
Thanks to Lemma (\ref{lem-4th}) and (\ref{lem-3rd}), equation (\ref{tag}) and (\ref{normal}) imply that 
$$\d_j\d_{\bar j} \Phi\in \cC^{,\a,\b}_0,\ \ \ |\z|^{2-2\b} \d_{\z}\d_{\bar\z} \Phi\in \cC^{,\a,\b}_0, $$
and the lemma follows.
\end{proof}

\begin{lemma}
\label{lem-vol2}
Let $\omega$ be a $\cC^{4,\a,\b}$ metric, and then we can estimate the growth rate of its volume form 
$$ |\z|^{2-2\b} \det g \in \cC^{2,\a,\b}.  $$
In particular, $|\z|^{2\b -2}g^{\bar\z\z} \in \cC^{2,\a,\b}$ and $ g^{\bar l k}\in \cC^{2,\a,\b}$ for all $k,l>1$.
\end{lemma}
\begin{proof}
Thanks to Lemma (\ref{lem-vol}), the estimate follows from the expansion formula of the volume form $\omega^n$, i.e.
\begin{equation}
\label{inverse}
\det g = (g_{\z\bar{\z}} - \sum_{k,l>1} (G_{\z\bar{\z}})^{\bar{l} k} g_{\z\bar{l}}g_{k\bar{\z}} ) \det G_{\z\bar{\z}},
\end{equation}
where $G_{\z\bar{\z}}$ is the complimentary matrix of $g$ with respect to $g_{\z\bar{\z}}$, and $\det G_{\z\bar\z}\in\cC^{2,\a,\b}$. 
Notice the conic metric $g$ is non-degenerate, and then we can assume $|\z|^{2-2\b} \det g>1$ and $\det G_{\z\bar\z} >1$ locally.
Therefore, the inverse matrix of $g$ can be expressed as 
$$ |\z|^{2\b-2}g^{\bar\z\z} = \frac{\det G_{\z\bar\z}}{|\z|^{2-2\b} \det g} \in\cC^{2,\a,\b}.  $$
For the tangential directions, we have 
$$ g^{\bar l k} = \frac{|\z|^{2-2\b}\det G_{k\bar l}}{|\z|^{2-2\b}\det g} \in \cC^{2,\a,\b}, $$
since the numerator $|\z|^{2-2\b}\det G_{k\bar l}$ is of class $\cC^{2,\a,\b}$ by the same reason with $|\z|^{2-2\b}\det g$.
\end{proof}

Take the Ricci potential of $\cC^{4,\a,\b}$ conic K\"ahler metric $\omega$ as 
$$\phi: = -\log \frac{|s|^{2-2\b}_h \omega^n}{\omega_0^n}. $$ 
Then we can estimate the growth rate of the Ricci curvature of $\omega$ near the divisor.
\begin{corollary}
\label{cor-ric}
We have
$$ \rho^{2-2\b}R_{\z\bar{\z}} \in \cC^{, \a,\b},\ \rho^{1-\b}R_{\z\bar l}\in \cC^{,\a,\b}_0,\ \rho^{1-\b} R_{k\bar\z} \in \cC^{,\a,\b}_0,\ R_{k\bar l}\in \cC^{,\a,\b}. $$
\end{corollary}
\begin{proof}
It is enough to prove $\ddbar \phi $ is of class $\cC^{,\a,\b}$ as a $(1,1)$ form,
and this is true if $e^{-\phi} \in \cC^{2,\a,\b}$ provided the non-degeneracy of the metric. But $e^{-\phi}$ is equal to $  h ( |\z|^{2-2\b} \det g ) $ locally,
where $h>0$ is a smooth function. Hence the corollary follows from Lemma (\ref{lem-vol2}). 
\end{proof}

\begin{lemma}
\label{lem-tangent}
For each fixed $j>1$, the curvature tensor 
$$ R_{\mu\bar\b j\bar j} $$ is cone admissible,
for all indices $\mu$ and $\b$. 
\end{lemma}
\begin{proof}
Note for $j>1$, the tensor $g_{j\bar j}$ is of class $\cC^{2,\a,\b}$, and all the third order terms like $\d_\mu g_{j\bar\q}, \d_{\bar\b}g_{\p\bar j}$ are cone admissible. 
Hence the lemma follows from the following 
\begin{equation}
R_{\mu\bar\b j\bar j} = \d_\mu\d_{\bar\b}g_{j\bar j} - g^{\bar\q\p} \d_\mu g_{j\bar\q} \d_{\bar\b} g_{\p\bar j}.
\end{equation}
\end{proof}

\begin{lemma}
\label{lem-normal}
The curvature tensor on the normal direction 
$$ R_{\mu\bar\b\z\bar\z} $$
is cone admissible for all indices $\mu,\b$.
\end{lemma}
\begin{proof}
There are four cases, which correspond to different choices of $\mu$ and $\b$:
\\
\\
\textbf{Case 1}
Let $\mu = \z, \b =\z$, and we have 
$$ R_{\z\bar\z\z\bar\z} = \d_{\z}\d_{\bar\z}g_{\z\bar\z} - g^{\bar\q\p} \d_{\z} g_{\z\bar\q} \d_{\bar\z}g_{\p\bar\z}.$$
However, we further have 
\begin{eqnarray}
\label{R001}
\d^{\phi}_{\z}\d^{\phi}_{\bar\z} g_{\z\bar\z} &=& |\z|^{2\b-2} \d_{\z} \{ \d_{\bar\z} (|\z|^{2-2\b} g_{\z\bar\z}) \}
\nonumber\\
& =& \d_{\z}\d_{\bar\z}g_{\z\bar\z} + (1-\b)\z^{-1} \d_{\bar\z} g_{\z\bar\z} 
\nonumber\\
&+& (1-\b) \bar\z^{-1}\d_\z g_{\z\bar\z} + (1-\b)^2 |\z|^{-2}g_{\z\bar\z}.
\end{eqnarray}
Moreover, 
\begin{eqnarray}
\label{R002}
g^{\bar\q\p} \d^{\phi}_{\z} g_{\z\bar\q}\d^{\phi}_{\bar\z}g_{\p\bar\z} &=& 
g^{\bar\q\p}(\d_\z g_{\z\bar\q} + (1-\b)\z^{-1}g_{\z\bar\q} ) (\d_{\bar\z} g_{\p\bar\z} + (1-\b) \bar\z^{-1} g_{\p\bar\z})
\nonumber\\
&=& g^{\bar\q\p} \d_\z g_{\z\bar\q} \d_{\bar\z}g_{\p\bar\z} + (1-\b)\z^{-1} \delta_{\p\z} \d_{\bar\z}g_{\p\bar\z}
\nonumber\\
&+& (1-\b) \bar\z^{-1} \delta_{\q\z} \d_{\z}g_{\z\bar\q} + (1-\b)^2 |\z|^{-2} \delta_{\p\z}g_{\p\bar\z}
\nonumber\\
&= & g^{\bar\q\p} \d_\z g_{\z\bar\q} \d_{\bar\z}g_{\p\bar\z} + (1-\b)\z^{-1} \d_{\bar\z} g_{\z\bar\z} 
\nonumber\\
&+& (1-\b) \bar\z^{-1}\d_\z g_{\z\bar\z} + (1-\b)^2 |\z|^{-2}g_{\z\bar\z}.
\end{eqnarray}
Subtract equation (\ref{R002}) from (\ref{R001}), and then we have  
$$ R_{\z\bar\z\z\bar\z} =  \d^{\phi}_{\z}\d^{\phi}_{\bar\z}g_{\z\bar\z} - g^{\bar\q\p} \d^{\phi}_{\z} g_{\z\bar\q} \d^{\phi}_{\bar\z}g_{\p\bar\z}.$$
Note $\d^{\phi}_{\z}\d^{\phi}_{\bar\z}g_{\z\bar\z}$ is cone admissible since $|\z|^{2-2\b} g_{\z\bar\z}$ is of class $\cC^{2,\a,\b}$, 
and all the third order terms $\d^{\phi}_{\z} g_{\z\bar\q}, \d^{\phi}_{\bar\z}g_{\p\bar\z} $ are cone admissible by Proposition (\ref{prop-chris}). 
\\
\\
\textbf{Case 2}
Let $\mu = k>1, \b =\z$, and we have
\begin{eqnarray}
\label{R003}
R_{k\bar\z\z\bar\z} &=& \d_k\d_{\bar\z}g_{\z\bar\z} - g^{\bar\q\p} \d_k g_{\z\bar\q} \d_{\bar\z} g_{\p\bar\z}
\nonumber\\
& =  & \d_k\d^{\phi}_{\bar\z}g_{\z\bar\z} - g^{\bar\q\p} \d_k g_{\z\bar\q} \d^{\phi}_{\bar\z} g_{\p\bar\z},
\end{eqnarray}
by a similar computation as in Case 1. Now $\d_k\d^{\phi}_{\bar\z}g_{\z\bar\z}$ is cone admissible since $|\z|^{2-2\b}g_{\z\bar\z}\in\cC^{2,\a,\b}$,
and the remaining third order terms are all cone admissible too. 
\\
\\
\textbf{Case 3}
Let $\mu = \z, \b= l>1$.
This case is similar to Case 2, so we do not repeat the argument here.
\\
\\
\textbf{Case 4}
Let $\mu = k, \b = l >1$, and we have
$$ R_{k\bar l \z\bar\z} = R_{\z\bar\z k\bar l} = \d_\z\d_{\bar\z} g_{k\bar l} - g^{\bar\q \p}\d_\z g_{k\bar\q}\d_{\bar\z} g_{\p\bar l}, $$
but $g_{k\bar l}$ is of class $\cC^{2,\a,\b}$, and all the third order terms like $\d_\z g_{k\bar\q}, \d_{\bar\z} g_{\p\bar l}$ 
are cone admissible by Proposition (\ref{prop-chris}).
\end{proof}

\begin{rem}
In fact, all the curvature tensors $R_{\nu\bar\lambda\gamma\bar\eta}$ are cone admissible by a similar calculation. 
\end{rem}

\section{conic cscK metrics}

In the first author's previous work \cite{Li}, three different notions of conic cscK metrics are discussed. 
Under the small angle condition, a slightly different notion will be introduced from the view of Monge-Amp\`ere equations.  
We will eventually see that it coincides with the so called \emph{strong conic cscK metrics} in \cite{Li}.  

Recall that $\Omega$ is the model conic K\"ahler metric, and we write 
$$\Omega: = \sum_{\a,\b} h_{\a\bar\b} dz_{\a}\wedge d\bar z_{\b}, $$
in local coordinate charts. All Christoffel symbols of $h$ are denoted by $\cu^{\mu}_{\gamma\a}$.
Now we introduce the following definition.

\begin{defn}
A $\cC^{2,\a,\b}$ conic cscK metric is a conic K\"ahler metric 
$\omega: = \Omega + dd^c u$ with $u\in\cC^{2,\a,\b}$ such that it satisfies the following Monge-Amp\`ere equation on $X-D$
\begin{equation}
\label{ma0}
(\Omega + dd^c u)^n = e^F \Omega^n,
\end{equation}
where the real-valued function $F$ is the solution of the following elliptic equation on $X-D$
\begin{equation}
\label{ma1}
\Delta_{\omega}F = tr_{\omega} Ric(\Omega) - c(\b),
\end{equation}
where $c(\b)$ is a topological constant. 
\end{defn}

Note here we adapt to the normalization 
$$ \int_{X-D} e^F \Omega^n = [\Omega]^n.$$
Based on these elliptic equations, we can lift the regularity of the conic cscK metric as follows.

\begin{theorem}
\label{thm-reg}
Suppose $u\in\cC^{2,\a,\b}$ is the potential of a conic cscK metric $\omega$. Then $u$ is of class $\cC^{4,\a,\b}$ under the small angle condition. 
\end{theorem}

Thanks to Lemma (\ref{lem-model}) and Corollary (\ref{cor-ric}), equation (\ref{ma1}) implies that the function $F$ is of class $\cC^{2,\a,\b}$ by Donaldson's estimate.
Writing equation (\ref{ma0}) as 
$$Ric(\omega) = Ric (\Omega) - dd^cF, $$
we conclude that the Ricci curvature of $\omega$ is cone admissible. 
Now we claim the conic cscK metric $\omega$ has the model growth in the 3rd. order,
and the proof is basically Brendle's computation.

\begin{lemma}
\label{lem-3-2}
We have 
$$ \d_{k}\d_{\bar{l}}\d_{i} u \in \cC^{,\a,\b},$$
$$ \rho^{1-\b} \d_k\d_{\bar{\z}}\d_{i} u \in \cC_0^{,\a,\b},$$
$$ \rho^{1-\b} \d_{\z}\d_{\bar{l}}\d_k u \in\cC_0^{,\a,\b},$$
$$ \rho^{2-2\b} \d_{\z}\d_{\bar{\z}}\d_k u \in \cC^{,\a,\b}, $$
$$ \rho^{2-2\b} \d^{\phi}_{\z}\d_{\bar{l}}\d_{\z} u = O(\rho^{\a\b}), $$
for all $k,l, i >1$. 
\end{lemma}
\begin{proof}
Differentiating equation (\ref{ma0}) in the tangential direction $z_k$, we have 
$$\Delta_{\omega}(\d_k u) = tr_{\Omega} (\d_k\Omega) + \d_k F - tr_{\omega}(\d_k\Omega). $$
For sufficient small $\a>0$, the $(1,1)$ form $ \d_k\Omega $ is of class $\cC^{,\a,\b}$. 
This implies $ tr_{\Omega} (\d_k \Omega)\in\cC^{,\a,\b}$ and $ tr_{\omega}(\d_k\Omega)\in\cC^{,\a,\b}$, since $\omega$ is uniformly equivalent to $\Omega$ and of class $\cC^{,\a,\b}$. 
Moreover, $\d_k F\in\cC^{,\a,\b}$ implies $\Delta_{\omega} (\d_k u) \in \cC^{,\a,\b}$. Therefore, first four lines of the lemma follow from $\d_k u\in\cC^{2,\a,\b}$. 

Now take $v = \d_{\bar l}u$, and then we have 
$$ |\z|^{2-2\b} \d_{\z}\d_{\bar\z} v \in \cC^{,\a,\b}.  $$
Hence we conclude the estimate 
$$\d^{\phi}_\z\d_\z\d_{\bar l}u = O(|\z|^{\a\b}),$$ 
thanks to Brendle's Appendix \cite{Br} again. 
\end{proof}

Let $D$ denote by the connection associated with $\Omega$, 
and we can further estimate all the third covariant derivatives.   

\begin{lemma}
\label{lem-3tang-2}
We have 
$$ (D^3 u) (\d_i, \d_{\bar l}, \d_k) \in \cC^{,\a,\b},$$
$$ \rho^{1-\b} (D^3 u) (\d_i, \d_{\bar\z}, \d_k) \in\cC^{,\a,\b}_0,$$
$$ \rho^{1-\b} (D^3 u) (\d_\z, \d_{\bar l}, \d_k) \in \cC^{,\a,\b}_0,$$
$$ \rho^{2-2\b} (D^3 u) (\d_\z, \d_{\bar\z}, \d_k) \in\cC^{,\a,\b},$$
for all $i, k, l >1$. Moreover, we have 
$$ \rho^{2-2\b} (D^3u) (\d_\z, \d_{\bar l}, \d_\z) = O(\rho^{\a\b}),$$
for all $l>1$. 
\end{lemma}
\begin{proof}
Note that $\Omega$ has the model growth in the 3rd. order, and then all the Christoffel symbols $\cuu^{\gamma}_{\a\mu}$ are cone admissible.
Therefore, the result follows directly from Lemma (\ref{lem-3-2}).
\end{proof}

\begin{lemma}
\label{lem-3normal-2}
We have
\begin{equation}
\label{1bar11}
 |\z|^{3-3\b}(D^3 u)(\d_\z, \d_{\bar\z}, \d_\z) \in\cC^{,\a,\b}_0.
\end{equation}
Moreover, we have 
\begin{equation}
\label{2bar11}
 \rho^{3-3\b}\d^{\phi}_{\z} u_{,\z\bar\z} \in \cC^{,\a,\b}_0. 
\end{equation}
\end{lemma}
\begin{proof}
Differentiating the Monge-Amp\`ere equation (\ref{ma0}) from the normal direction gives 
\begin{equation}
\label{ma-normal}
 g^{\bar\b\a} (D^3 u) (\d_\z, \d_{\bar\b}, \d_\a) = \d_\z F.
\end{equation}
Thanks to Lemma (\ref{lem-3tang-2}), we obtain
$$ \rho^{1-\b} g^{\bar l k} (D^3 u) (\d_\z, \d_{\bar l}, \d_k ) \in\cC^{,\a,\b}_0, $$
$$ \rho^{1-\b} g^{\bar l\z} (D^3 u) (\d_\z, \d_{\bar l}, \d_\z) \in\cC^{,\a,\b}_0, $$
$$ \rho^{1-\b} g^{\bar\z k} (D^3u) (\d_\z, \d_{\bar\z}, \d_k) \in \cC^{,\a,\b}_0 $$
for $k,l>1$. This implies 
$$ |\z|^{1-\b} g^{\bar\z\z} (D^3u) (\d_\z, \d_{\bar\z}, \d_\z) \in \cC^{,\a,\b}_0, $$
since $\rho^{1-\b} \d_\z F \in \cC^{,\a,\b}_0$. 
Then we obtain equation (\ref{1bar11}), and equation (\ref{2bar11}) follows by the computation 
$$ |\z|^{3-3\b} (D^3 u) (\d_\z,\d_{\bar\z}, \d_\z) = |\z|^{3-3\b} (\d^{\phi}_\z u_{,\z\bar\z} + \cuu^{\gamma}_{\z\z} u_{,\gamma\bar\z}). $$
\end{proof}
Then by the standard computation, we proved the following 
\begin{corollary}
A $\cC^{2,\a,\b}$ conic cscK metric $\omega$ has the model growth in the 3rd. order.
\end{corollary}

In order to prove $ \Delta_{\Omega_0}u \in \cC^{2,\a,\b}$, it is enough to prove $u_{,j\bar j}\in\cC^{2,\a,\b}$ and $|\z|^{2-2\b}u_{,\z\bar\z}\in\cC^{2,\a,\b}$.
And it is comparatively easy to control the growth in tangential directions.
\begin{lemma}
\label{lem-tang-2}
We have 
$$ \d_k\d_{\bar l}u  \in\cC^{2,\a,\b},$$ for all $j>1$. 
\end{lemma}
\begin{proof}
Recall that all the Ricci curvature of $\omega$ is of class $\cC^{,\a,\b}$, and the Ricci tensor can be written as 
$$ R_{k\bar l} = \Delta_{\omega} g_{k\bar l} - g^{\bar\b\a}g^{\bar\nu\mu}\d_k g_{\a\bar\nu} \d_{\bar l} g_{\mu\bar\b}\in\cC^{,\a,\b}. $$
Therefore we have 
$$ \Delta_{\omega} g_{k\bar l} \in \cC^{,\a,\b},$$
since all the three tensors like $\d_k g_{\a\bar\nu}, \d_{\bar l} g_{\mu\bar\b}$ are cone admissible thanks to Lemma (\ref{lem-3-2}). 
Then the result follows from Donaldson's estimate. 
\end{proof}

\begin{lemma}
\label{lem-normal-2}
We have 
$$ |\z|^{2-2\b} \d_\z\d_{\bar\z} u \in \cC^{2,\a,\b}.$$
\end{lemma}
\begin{proof}
It is enough to prove $\Delta_{\omega} (|\z|^{2-2\b} \d_\z\d_{\bar\z} u) \in\cC^{,\a,\b}$. According to equation (\ref{ma-normal}), we have 
\begin{equation}
\label{ma-nor}
|\z|^{2-2\b}\nabla_{\z} (g^{\bar\b\a} u_{,\bar\z\a\bar\b}) = |\z|^{2-2\b}\d_\z\d_{\bar\z}F\in\cC^{,\a,\b}, 
\end{equation}
where $\nabla$ denotes the connection associated with metric $\Omega$ here. 
Then we obtain 
$$ \nabla_{\z} (g^{\bar\b\a} u_{,\bar\z\a\bar\b}) = (\nabla_\z g^{\bar\b\a}) u_{,\bar\z\a\bar\b} + g^{\bar\b\a} \nabla_\z\nabla_{\bar\z} u_{,\a\bar\b}. $$
Thanks to Lemma (\ref{lem-3tang-2}) and (\ref{lem-3normal-2}), the first term is cone admissible since it can be written as 
$$ \nabla_\z g^{\bar\b\a} = - g^{\bar\b\mu} g^{\bar\nu\a} \nabla_{\z} u_{,\mu\bar\nu}.$$
On the other hand, the second term can be computed as (compare with equation (\ref{005}))

\begin{eqnarray}
\label{00555}
g^{\bar\b\a}\nabla_{\z}\nabla_{\bar\z} u_{,\a\bar\b} &=& g^{\bar\b\a} \nabla^{\phi}_{\z}\nabla^{\phi}_{\bar\z} u_{,\a\bar\b}
\nonumber\\
&=&   g^{\bar\z\a}\nabla^{\phi}_{\z} ( \d^{\phi}_{\bar\z}u_{,\a\bar\z} - \cuu^{\bar\eta}_{\bar\z\bar\z} u_{,\a\bar\eta} ) 
+   \sum_{l>1}g^{\bar l\a}\nabla^{\phi}_{\z} ( \d_{\bar\z}u_{,\a\bar l} - \cu^{\bar\eta}_{\bar\z\bar l} u_{,\a\bar\eta} ) 
\nonumber\\
&=& g^{\bar\z\z} ( \d^{\phi}_{\z}\d^{\phi}_{\bar\z}u_{,\z\bar\z} - (\d_{\z}\cuu^{\bar\eta}_{\bar\z\bar\z}) u_{,\z\bar\eta} - \cuu^{\g}_{\z\z} \d^{\phi}_{\bar\z} u_{,\g\bar\z} 
- \cuu^{\bar\eta}_{\bar\z\bar\z} \d^{\phi}_{\z} u_{,\z\bar\eta} + \cuu^{\g}_{\z\z}\cuu^{\bar\eta}_{\bar\z\bar\z} u_{,\g\bar\eta}) 
\nonumber\\
&+& \sum_{k>1}g^{\bar\z k} (\d_\z\d^{\phi}_{\bar\z}u_{,k\bar\z} - (\d_\z\cuu^{\bar\eta}_{\bar\z\bar\z})u_{,k\bar\eta} - \cuu^{\bar\eta}_{\bar\z\bar\z}\d_\z u_{,k\bar\eta}    
- \cu^{\g}_{\z k} \d^{\phi}_{\bar\z} u_{,\g\bar\z} + \cu^{\g}_{\z k}\cuu^{\bar\eta}_{\bar\z\bar\z} u_{,\g\bar\eta}   )
\nonumber\\
&+& \sum_{l>1} g^{\bar l\z} (\d^{\phi}_{\z}\d_{\bar\z}u_{,\z\bar l} - (\d_{\z}\cu^{\bar\eta}_{\bar\z\bar l})u_{,\z\bar\eta} - \cu^{\bar\eta}_{\bar\z\bar l}\d^{\phi}_{\z}u_{,\z\bar\eta} 
- \cuu^{\g}_{\z\z}\d_{\bar\z}u_{,\g\bar l} + \cuu^{\g}_{\z\z}\cu^{\bar\eta}_{\bar\z\bar l} u_{,\g\bar\eta}  )
\nonumber\\
&+ & \sum_{k,l>1}g^{\bar lk}(\d_{\z}\d_{\bar\z} u_{,k\bar l} - (\d_{\z}\cu^{\bar\eta}_{\bar\z\bar l})u_{,k\bar\eta} - \cu^{\bar\eta}_{\bar\z\bar l} \d_{\z}u_{,k\bar l}
- \cu^{\g}_{\z k}\d_{\bar\z}u_{,\g\bar l} + \cu^{\g}_{\z k}\cu^{\bar\eta}_{\bar\z\bar l} u_{,\g\bar\eta}).
\nonumber\\
&=& |\z|^{2\b-2} \Delta_{\omega} (|\z|^{2-2\b} \d_\z\d_{\bar\z} u) + \emph{two and three tensors}.
\nonumber\\
\end{eqnarray}
Note that the derivatives of the Christoffel symbols are exactly curvature tensors 
$$ \d_{\z}\cuu^{\bar\eta}_{\bar\z \bar \nu} = \d_{\z}\cu^{\bar\eta}_{\bar\z \bar \nu} = R^{\bar\eta}_{\ \bar\z\z\nu}. $$
These curvature tensors are cone admissible by Lemma (\ref{lem-normal}), and all the two and three tensors in $u$ are also cone admissible thanks to Lemma (\ref{lem-3-2}), (\ref{lem-3normal-2}).
Therefore, the result follows.
\end{proof}

\begin{proof}[Proof of Theorem (\ref{thm-reg})]
Combining with Lemma (\ref{lem-tang-2}) and (\ref{lem-normal-2}), we have 
$$ \Delta_{\omega} (\Delta_{\Omega_0} u) \in\cC^{,\a,\b}.$$
And the result follows. 
\end{proof}

\section{Reductivity }
Let $Aut(X;D)$ be the group of holomorphic automorphism of $X$ which fixes the divisor $D$. 
There is a one-one correspondence between $Aut(X;D)$ and the space $\mathfrak{h}(X;D)$ of all holomorphic vector fields tangential to the divisor. 
Note that $\mathfrak{h}(X;D)$ is a Lie algebra since we have 
$$ [\mathfrak{h}(X;D), \mathfrak{h}(X;D)] = \mathfrak{h}(X;D).$$
Then we can generalize the Lichnerowicz-Calabi theorem to the conic setting. 

\begin{theorem}
\label{thm-reductive}
Suppose $\omega$ is a $\cC^{2,\a,\b}$ conic cscK metric on $X$ under the small angle condition. 
Then the Lie algebra $\mathfrak{h} (X;D)$ has a semidirect sum decomposition:
\begin{equation}
\mathfrak{h}(X;D) = \mathfrak{a}(X;D) \oplus \mathfrak{h'} (X;D),
\end{equation}
where $\mathfrak{a}(X;D)$ is the complex Lie subalgebra of $\mathfrak{h}(X;D)$ consisting of all parallel holomorphic vector fields tangential to $D$,
and $\mathfrak{h'}(X;D)$ is an ideal of $\mathfrak{h}(X;D)$ consisting of the image under $grad_g$ of the kernel of $\cD$ operator.

Furthermore $\mathfrak{h}'(X;D)$ is the complexification of a Lie algebra consisting of Killing vector fields of $X$ tangential to $D$. 
In particular $\mathfrak{h}'(X;D)$ is reductive. 
\end{theorem}

\subsection{Hodge decomposition}
Let $\omega$ be a $\cC^{4,\a,\b}$ conic cscK metric, and $Y$ be a holomorphic vector field on $X$ tangential to $D$. 
The induced $(0,1)$ form of $Y$ by $\omega$ is 
$$ \tau: = \downarrow^{\omega} Y, $$
and
$$  \tau_{\bar{\z}} = g_{\z\bar{\z}} Y^{\z} + \sum_{k>1}g_{k\bar{\z}}Y^{k}, $$
$$  \tau_{\bar{l}} = g_{\z\bar{l}} Y^{\z} + \sum_{k>1} g_{k\bar{l}}Y^{k}. $$
As a $(0,1)$ form, $\tau$ is of class $\cC^{,\a,\b}$ since $\rho^{1-\b}\tau_{\bar{\z}}\in \cC^{,\a,\b}_0 $ and $\tau_{\bar{l}}\in \cC^{,\a,\b}$,
and it is easy to see $\tau$ is closed on $X-D$ by a local computation \cite{Fut}.

Let $\vartheta$ be the formal adjoint of $\dbar$ operator acting on functions with respect to the metric $\omega$. 
It is not clear that this operator $\vartheta$ has closed range, but we can compute it outside the divisor. 
\begin{equation}
\label{oneform}
\vartheta\tau = g^{\bar{\z}\z}\nabla_{\z}\tau_{\bar\z} + \sum_{k>1}g^{\bar{\z}k}\nabla_{k}\tau_{\bar{\z}}
+ \sum_{l,k>1} ( g^{\bar{l}\z}\nabla_{\z}\tau_{\bar{l}} + g^{\bar{l}k}\nabla_{k}\tau_{\bar{l}}).
\end{equation}
Fortunately, we can estimate their growth rate near the divisor thanks to the model growth condition.
\begin{eqnarray}
\label{oneform-1}
g^{\bar{\z}\z}\nabla_{\z}\tau_{\bar\z} &=& g^{\bar{\z}\z} (g_{\z\bar{\z}}\nabla_{\z} Y^{\z} + \sum_{k>1}g_{k\bar{\z}}\nabla_{\z}Y^{k})   
\nonumber\\
&= & g^{\bar{\z}\z} g_{\z\bar{\z}} (\d_{\z} Y^{\z} + \Gamma^{\z}_{\z\z} Y^{\z} + \sum_{k>1} \Gamma^{\z}_{\z k} Y^k)
\nonumber\\
&+& \sum_{k>1} g^{\bar{\z}\z} g_{k\bar{\z}} (\d_{\z} Y^k + \Gamma^{k}_{\z\z}Y^{\z} + \sum_{i>1} \Gamma^{k}_{\z i} Y^i).
\end{eqnarray}
Since $Y^{\z}$ is a holomorphic function vanishing on $\{\z= 0 \}$, we can write $Y^{\z}$ locally as
$$ Y^{\z} = \z\cdot f, $$
where $f$ is another local holomorphic function near a point $p\in D$.
Then we can infer that 
$$g^{\bar{\z}\z}\nabla_{\z}\tau_{\bar\z}  \in \cC^{,\a,\b}. $$
For any fixed $k, l>1$, we have similar estimates.
\begin{eqnarray}
\label{oneform-2}
g^{\bar{\z}k}\nabla_{k}\tau_{\bar{\z}} &=& g^{\bar{\z}k}(g_{\z\bar{\z}} \nabla_{k} Y^{\z} + \sum_{i>1} g_{i\bar{\z}}\nabla_k Y^{i})
\nonumber\\
&=& g^{\bar{\z}k}g_{\z\bar{\z}} (\d_k Y^{\z} + \Gamma^{\z}_{k\z} Y^{\z} + \sum_{i>1}\Gamma^{\z}_{ki} Y^i)
\nonumber\\
&+& \sum_{i>1} g^{\bar{\z}k} g_{i\bar{\z}} (\d_{k} Y^i + \Gamma^{i}_{k\z}Y^{\z} + \sum_{m>1} \Gamma_{km}^i Y^m),
\end{eqnarray}
and 
\begin{eqnarray}
\label{oneform-3}
g^{\bar{l}\z} \nabla_{\z} \tau_{\bar l} &=& g^{\bar{l}\z} (g_{\z\bar{l}} \nabla_{\z}Y^{\z} + \sum_{i>1} g_{i\bar l} \nabla_{\z}Y^i )
\nonumber\\
&=& g^{\bar{l}\z} g_{\z\bar{l}} (\d_{\z}Y^{\z} + \Gamma^{\z}_{\z\z} Y^{\z} +\sum_{i>1} \Gamma^{\z}_{\z i} Y^i)
\nonumber\\
&+& \sum_{i>1}g^{\bar{l}\z} g_{i\bar{l}}(\d_{\z}Y^i  + \Gamma^{i}_{\z\z}Y^{\z} + \sum_{m>1}\Gamma^{i}_{\z m} Y^m).
\end{eqnarray}
Moreover,
\begin{eqnarray}
\label{oneform-4}
g^{\bar{l}k}\nabla_{k} \tau_{\bar{l}} &=&  g^{\bar{l}k} \d_{k}\tau_{\bar{l}}
\nonumber\\
&= & g^{\bar{l}k} (\d_{k}g_{\z\bar l} Y^{\z} + g_{\z\bar l} \d_k Y^{\z} ) + \sum_{j>1}g^{\bar lk} (\d_kg_{j\bar l}Y^{j} + g_{j\bar l}\d_k Y^{j}).
\end{eqnarray}
Notice that $\d_k Y^{\z} = \z \d_k f$ still vanishes along the divisor at least up to the first order. 
We conclude from equations (\ref{oneform-2}), (\ref{oneform-3}) and (\ref{oneform-4}) that 
$$ g^{\bar{\z}k}\nabla_k \tau_{\bar{\z}} \in \cC^{,\a,\b},\ \ g^{\bar{l}\z}\nabla_{\z}\tau_{\bar{l}} \in \cC^{,\a,\b}\ \ g^{\bar{l}k}\nabla_{k} \tau_{\bar{l}}\in \cC^{,\a,\b}. $$

In other words, the function $\vartheta\tau$ is in the class $\cC^{,\a,\b}$. Now consider the following Laplacian equation:
\begin{equation}
\label{lap1}
\Delta_{\omega} u = \vartheta\tau. 
\end{equation}
Notice that $\Delta_{\omega}$ is a real operator.
Thanks to the linear theory established in \cite{CZ}, there exists a weak solution $u\in W^{1,2}(\omega)$ satisfying equation (\ref{lap1}). 
Then by Donaldson's estimate (Theorem 2.5, \cite{Br}), we conclude that $u\in \cC^{2,\a,\b}$. 
Indeed, we can lift its regularities to a higher order.
\begin{prop}
\label{prop-reg}
The solution $u$ of the Lapalacian equation $$\Delta_{\omega}u = \vartheta \tau$$ is of class $\cC^{4,\a,\b}$.
\end{prop}
In fact, the function $u$ is in the kernel of the so called Lichnerowicz operator $\cD$ of the conic cscK metric $\omega$ (see Section \ref{sub-2}) outside of the divisor.
This is because the projection $\gamma: = \tau - \dbar u$ is a $\dbar$-harmonic $(0,1)$ form, 
and we can utilize equation (\ref{new4-001}). 
Hence we have 
$$\Delta^2_{\omega} u = - R^{\bar\b\a}u_{,\a\bar\b} \in \cC^{,\a,\b}, $$
on $X-D$. Now thanks to Donaldson's estimate and the fact that $\cD$ is a real operator, we further have 
\begin{equation}
\label{new-4-4}
\Delta_{\omega}  u\in \cC^{2,\a,\b}.
\end{equation}
And note that $u$ is smooth outside of the divisor by standard elliptic estimates.
Then Proposition (\ref{prop-reg}) follows directly from the following rough Schauder estimate for the operator $\cD$
\begin{prop}
\label{prop-sch}
Let $\omega$ be a $\cC^{4,\a,\b}$ conic cscK metric, and $\cD$ be the Lichnerowicz operator associated to $\omega$.
Suppose that $u$ is of class $\cC^{2,\a,\b}\cap C^{4}(X-D)$ and $f$ is of class $\cC^{,\a,\b}$ 
such that $\cD u = f$ away from the divisor. Then $u$ is of class $\cC^{4,\a,\b}$. 
\end{prop}
 
Note that $\Delta_{\omega}u$ is again of class $\cC^{2,\a,\b}$ by the same argument with above, and we will prove that this condition (equation \ref{new-4-4}) implies that $u$ is of class $\cC^{4,\a,\b}$. 
 
\begin{lemma}
\label{lem-D003}
For each fixed $j>1$, we have
$$\d_j u \in \cC^{2,\a,\b}.$$
\end{lemma}
\begin{proof}
Take the one derivative in $k$ direction with repect to the Laplacian equation
$$ \d_j (\Delta_{\omega} u) = g^{\bar\b\a}\d_j u_{,\a\bar\b} - g^{\bar\b\a}(\Gamma^{\gamma}_{j\a} u_{,\gamma\bar\b}) \in \cC^{,\a,\b}.$$
Since the tensor $\Gamma^{\g}_{j\a}$ is cone admissible for all $\g, \a$, we have 
$$ \Delta_{\omega} (\d_j u) = g^{\bar\b\a}\d_j u_{,\a\bar\b} \in \cC^{,\a,\b}, $$ 
and the lemma follows from Donaldson's estimate. 
\end{proof}

\begin{lemma}
\label{lem-D004}
For each $j>1$, we have 
$$ \d_j\d_{\bar j} u \in \cC^{2,\a,\b}.$$
\end{lemma}
\begin{proof}
It's enough to prove that $\Delta_{\omega}(\d_j\d_{\bar j} u)$ is of class $\cC^{,\a,\b}$,
and we have 
\begin{eqnarray}
\label{four}
 \d_j\d_{\bar j} (\Delta_{\omega} u) &=& g^{\bar\b\a} \nabla_j\nabla_{\bar j} u_{,\a\bar\b}
 \nonumber\\
 &=& g^{\bar\b\a} (\d_j\d_{\bar j}u_{,\a\bar\b} - \Gamma^{\g}_{j\a} \d_{\bar j} u_{,\g\bar\b} - \Gamma^{\bar\eta}_{\bar j\bar\b} \d_{j} u_{,\a\bar\eta}
 \nonumber\\
 &-&  ( \d_j \Gamma^{\bar\eta}_{\bar j\bar\b} ) u_{,\a\bar\eta} + \Gamma^{\bar\eta}_{\bar j\bar\b} \Gamma^{\g}_{j\a} u_{,\g\bar\eta})
 \nonumber\\
 & = & \Delta_{\omega}(u_{,j\bar j}) - g^{\bar\b\a}\Gamma^{\g}_{j\a} \d_{\bar j} u_{,\g\bar\b} - g^{\bar\b\a}\Gamma^{\bar\eta}_{\bar j\bar\b} \d_{j} u_{,\a\bar\eta}
 \nonumber\\
 & - & R^{\bar\eta\ \ \a}_{\ \bar j j}u_{,\a\bar\eta} + g^{\bar\b\a}\Gamma^{\bar\eta}_{\bar j\bar\b} \Gamma^{\g}_{j\a} u_{,\g\bar\eta}
\end{eqnarray}
Note that all third order terms $\d_{\bar j} u_{,\g\bar\b},  \d_{j} u_{,\a\bar\eta}$ are cone admissible by Lemma (\ref{lem-D003}),
and all Chritoffel symbols with an index $j$ and all curvature tensors $R_{\mu\bar j j\bar\b}$ are cone admissible by Lemma (\ref{lem-tangent}). 
Therefore, the result follows since $\d_j\d_{\bar j} (\Delta_{\omega} u) \in \cC^{,\a,\b}$.
\end{proof}

\begin{lemma}
\label{lem-D005}
The following three tensors 
$$ \d_{\z} u_{,k\bar\eta},\ \ \d^{\phi}_{\z} u_{,\z\bar l},\ \ \d^{\phi}_{\z} u_{,\z\bar\z}  $$
for $k, l>1$ and all $\eta$ are cone admissible.
\end{lemma}
\begin{proof}
First, all three tensors $$ \d_{\z} u_{,k\bar\eta} = \d_{\z}\d_{\bar\eta} u_{,k}$$ are cone admissible by Lemma (\ref{lem-D003}). 
It also implies that $$ |\z|^{2-2\b} \d_{\z}\d_{\bar\z} u_{,\bar l} \in \cC^{,\a,\b},$$ for all $l>1$. 
Thanks to Brendle's estimate and small angle condition, we have 
$$ |\z|^{2-2\b}\d^{\phi}_{\z}\d_{\z} u_{,\bar l} = |\z|^{2-2\b}(\d_{\z}\d_{\z} u_{,\bar l} + (1-\b)\d_\z u_{,\bar l}) \in \cC^{,\a,\b}_0. $$
Now we can compute 
\begin{eqnarray}
\label{D555}
\d_\z (\Delta_{\omega} u) &=& g^{\bar\b\a} \nabla^{\phi}_{\z} u_{,\a\bar\b}
\nonumber\\
& = & g^{\bar\z\z} ( \d^{\phi}_{\z} u_{,\z\bar\z} - \tgg^{\g}_{\z\z} u_{,\g\bar\z}  ) + g^{\bar l \z} ( \d^{\phi}_{\z} u_{,\z\bar l} - \tgg^{\g}_{\z\z} u_{,\g\bar l }  )
\nonumber\\
& +& g^{\bar\b k} (\d_\z u_{,k\bar\b} - \Gamma^{\g}_{\z k} u_{,\g\bar\b}).
\end{eqnarray}
Now all the Christoffel symbols are cone admissible, and all third order terms $\d_\z u_{k\bar\b}, \d^{\phi}_{\z} u_{,\z\bar l} $ are cone admissible.
It follows that $g^{\bar\z\z} \d^{\phi}_{\z} u_{,\z\bar\z}$ is cone admissible, since $|\z|^{1-\b}\d_{\z}(\Delta_{\omega} u) \in \cC^{,\a,\b}_0$. 
That is to say, the function $ v: = |\z|^{1-\b} g^{\bar\z\z} \d^{\phi}_{\z} u_{,\z\bar\z}$ is of class $\cC^{,\a,\b}$, and it implies that the last tensor 
$$ |\z|^{3-3\b} \d^{\phi}_{\z} u_{,\z\bar\z} =   v \cdot  \{ |\z|^{2-2\b}(g_{\z\bar{\z}} - \sum_{k,l>1} (G_{\z\bar{\z}})^{\bar{l} k} g_{\z\bar{l}}g_{k\bar{\z}} ) \}$$
is of class $\cC^{,\a,\b}$ by Lemma (\ref{lem-vol2}). 
\end{proof}

\begin{lemma}
\label{lem-D006}
We have 
$$ |\z|^{2-2\b} \d_\z\d_{\bar\z}u \in \cC^{2,\a,\b}.$$
\end{lemma}
\begin{proof}
It's enough to prove $\Delta_{\omega}(|\z|^{2-2\b} \d_\z\d_{\bar\z}u)$ is a $\cC^{,\a,\b}$ function. 
But we have 
\begin{eqnarray}
\label{005}
\d_{\z}\d_{\bar\z} (\Delta_{\omega} u) &=& g^{\bar\b\a} \nabla^{\phi}_{\z}\nabla^{\phi}_{\bar\z} u_{,\a\bar\b}
\nonumber\\
&=&   g^{\bar\z\a}\nabla^{\phi}_{\z} ( \d^{\phi}_{\bar\z}u_{,\a\bar\z} - \tgg^{\bar\eta}_{\bar\z\bar\z} u_{,\a\bar\eta} ) 
+   \sum_{l>1}g^{\bar l\a}\nabla^{\phi}_{\z} ( \d_{\bar\z}u_{,\a\bar l} - \tg^{\bar\eta}_{\bar\z\bar l} u_{,\a\bar\eta} ) 
\nonumber\\
&=& g^{\bar\z\z} ( \d^{\phi}_{\z}\d^{\phi}_{\bar\z}u_{,\z\bar\z} - (\d_{\z}\tgg^{\bar\eta}_{\bar\z\bar\z}) u_{,\z\bar\eta} - \tgg^{\g}_{\z\z} \d^{\phi}_{\bar\z} u_{,\g\bar\z} 
- \tgg^{\bar\eta}_{\bar\z\bar\z} \d^{\phi}_{\z} u_{,\z\bar\eta} + \tgg^{\g}_{\z\z}\tgg^{\bar\eta}_{\bar\z\bar\z} u_{,\g\bar\eta}) 
\nonumber\\
&+& \sum_{k>1}g^{\bar\z k} (\d_\z\d^{\phi}_{\bar\z}u_{,k\bar\z} - (\d_\z\tgg^{\bar\eta}_{\bar\z\bar\z})u_{,k\bar\eta} - \tgg^{\bar\eta}_{\bar\z\bar\z}\d_\z u_{,k\bar\eta}    
- \tg^{\g}_{\z k} \d^{\phi}_{\bar\z} u_{,\g\bar\z} + \tg^{\g}_{\z k}\tgg^{\bar\eta}_{\bar\z\bar\z} u_{,\g\bar\eta}   )
\nonumber\\
&+& \sum_{l>1} g^{\bar l\z} (\d^{\phi}_{\z}\d_{\bar\z}u_{,\z\bar l} - (\d_{\z}\tg^{\bar\eta}_{\bar\z\bar l})u_{,\z\bar\eta} - \tg^{\bar\eta}_{\bar\z\bar l}\d^{\phi}_{\z}u_{,\z\bar\eta} 
- \tgg^{\g}_{\z\z}\d_{\bar\z}u_{,\g\bar l} + \tgg^{\g}_{\z\z}\tg^{\bar\eta}_{\bar\z\bar l} u_{,\g\bar\eta}  )
\nonumber\\
&+ & \sum_{k,l>1}g^{\bar lk}(\d_{\z}\d_{\bar\z} u_{,k\bar l} - (\d_{\z}\tg^{\bar\eta}_{\bar\z\bar l})u_{,k\bar\eta} - \tg^{\bar\eta}_{\bar\z\bar l} \d_{\z}u_{,k\bar l}
- \tg^{\g}_{\z k}\d_{\bar\z}u_{,\g\bar l} + \tg^{\g}_{\z k}\tg^{\bar\eta}_{\bar\z\bar l} u_{,\g\bar\eta}).
\nonumber\\
\end{eqnarray}
Observe that all the third tensors appeared in equation (\ref{005}) are cone admissible by Lemma (\ref{lem-D005}) and Remark (\ref{rem-chirs}).
Moreover, since we have 
$$\d_{\z}\tgg^{\bar\z}_{\bar\z\bar\z} = \d_{\z}\tg^{\bar\z}_{\bar\z\bar\z},$$
all tensors like $\d_{\z}\tgg^{\bar\eta}_{\bar\z\bar\b}$ are cone 
admissible for any $\b$ thanks to Lemma (\ref{lem-normal}). 

Then the remaining fourth tensors in equation (\ref{005}) are the following 
\begin{equation}
\label{006}
|\z|^{2-2\b}g^{\bar\z\z}\d^{\phi}_{\z}\d^{\phi}_{\bar\z}u_{,\z\bar\z} = g^{\bar\z\z} \d_{\z}\d_{\bar\z}(|\z|^{2-2\b} u_{,\z\bar\z} ),
\end{equation}
\begin{equation}
|\z|^{2-2\b} g^{\bar\z k} \d_{\z}\d^{\phi}_{\bar\z} u_{,k\bar\z} = |\z|^{2-2\b}g^{\bar\z k}\d_k\d^{\phi}_{\bar\z}u_{,\z\bar\z}=g^{\bar\z k} \d_k\d_{\bar\z} (|\z|^{2-2\b} u_{,\z\bar\z}),
\end{equation}
\begin{equation}
|\z|^{2-2\b} g^{\bar l\z} \d^{\phi}_{\z}\d_{\bar\z} u_{,\z\bar l} = g^{\bar l\z} \d_{\bar l}\d_{\z} (|\z|^{2-2\b} u_{,\z\bar\z}),
\end{equation}
\begin{equation}
|\z|^{2-2\b} g^{\bar l k} \d_{\z}\d_{\bar\z} u_{,k\bar l} = g^{\bar l k} \d_k\d_{\bar l} u_{,\z\bar\z}.
\end{equation}
Hence we have
$$ \Delta_{\omega} ( |\z|^{2-2\b}u_{,\z\bar\z} ) = g^{\bar\b\a} \d_\a\d_{\bar\b}(|\z|^{2-2\b} u_{,\z\bar\z}) \in \cC^{,\a,\b}, $$
and the result follows.
\end{proof}

\begin{proof}[Proof of Proposition (\ref{prop-sch})]
Take Laplacian on $\Delta_{\Omega_0}u$ with respect to the $\cC^{4,\a,\b}$ conic cscK metric $\omega$, and we have 
$$\Delta_{\omega}(\Delta_{\Omega_0} u) = \Delta_{\omega}(\b^{-2} |\z|^{2-2\b}u_{,\z\bar\z}) + \sum_{j>1} \Delta_{\omega} u_{,j\bar j} \in\cC^{,\a,\b}, $$
thanks to Lemma (\ref{lem-D004}) and (\ref{lem-D006}). Then the result follows from Donaldson's estimate.
\end{proof}

Finally if we take $ \gamma = \tau - \dbar u$, then $\gamma$ is a $(0,1)$ form of class $\cC^{,\a,\b}$ too. 
Moreover, $\gamma$ is $\dbar$-harmonic outside the divisor, in the sense that 
$$\dbar\gamma =0; \ \ \ \vartheta\gamma =0, $$
on $X-D$. Therefore, we proved the following lemma. 
\begin{lemma}
\label{lem-hodge}
Suppose $\tau$ is the lifting of a holomorphic vector field $Y$ tangential to the divisor by the metric $\omega$. It can be decomposed into two parts
\begin{equation}
\label{hodge}
\tau = \gamma + \dbar u,
\end{equation}
where $u\in \cC^{4,\a,\b}$, $\gamma\in \cC^{,\a,\b}$ and $\gamma$ is $\dbar$-harmonic outside $D$. 
\end{lemma}
\subsection{Lichnerowicz operator}
\label{sub-2}
For a fixed smooth K\"ahler metric $\omega_0$, we denote by $L_{\omega_0}: C^{\infty}_{\mathbb{C}}(X)\rightarrow C^{\infty}(T^{1,0}X\otimes T^{*0,1}(X))$ the second order differential operator (not real!) 
$$L_{\omega_0} u = \nabla_{\bar{\a}}\nabla^{\mu} u \frac{\d}{\d z^{\mu}} \otimes d\bar{z}^{\a}.$$
There is another 4th. order elliptic differential operator $D_{\omega_0}: C^{\infty}_{\mathbb{C}}(X)\rightarrow C^{\infty}_{\mathbb{C}}(X)$ defined as 
$$D_{\omega_0} u = \Delta_{\omega_0}^2 u +  \nabla_{\mu}(R^{\bar{\a}\mu}\nabla_{\bar{\a}} u).$$
It is standard to see that $$D_{\omega_0} = L^*_{\omega_0}L_{\omega_0},$$ where $L^*_{\omega_0}$ is the formal adjoint of $L_{\omega_0}$ with respect to the $L^2$-inner product.

Now we want to consider these operators in the conic setting. Recall that $\omega$ is a $\cC^{2,\a,\b}$ conic cscK metrics. 
Then it is indeed of class $\cC^{4,\a,\b}$ according to Theorem (\ref{thm-reg}).
We define $$\cD: \cC^{4,\a,\b} \rightarrow \cC^{,\a,\b}$$ as a 4th. order differential operator 
$$ \cD u =  \Delta_{\omega}^2 u +  R^{\bar{\nu}\mu}\nabla_{\mu}\nabla_{\bar{\nu}} u,$$
where the covariant derivative $\nabla$ and curvature $R$ are all respect to the conic cscK metric $\omega$. 
\begin{lemma}
\label{lem-lich}
$\cD$ is a well defined, real operator. 
\end{lemma}
\begin{proof}
We first claim that $\Delta_{\omega}u \in \cC^{2,\a,\b}$. Then $\Delta_{\omega}^2$ is an operator from $\cC^{4,\a,\b}$ to $\cC^{,\a,\b}$ by direct computation 
\begin{equation}
\label{lap}
\Delta_{\omega}^2 u = g^{\bar{\z}\z} (\Delta_{\omega}u)_{,\z\bar{\z}} + \sum_{k,l>1}  g^{\bar{\z}k} (\Delta_{\omega}u)_{,k\bar{\z}} + g^{\bar{l}\z}(\Delta_{\omega} u)_{,\z\bar{l}}
+ g^{\bar{l}k}(\Delta_{\omega} u)_{,k\bar{l}}.
\end{equation}
To prove the claim, we can compute as follows
\begin{eqnarray}
\label{lap001}
\Delta_{\omega}u &=& g^{\z\bar{\z}} u_{,\z\bar{\z}} + \sum_{k,l>1} g^{\bar{\z}k}u_{,k\bar{\z}} + g^{\z\bar{l}} u_{,\z\bar{l}} + g^{\bar{l}k} u_{,k\bar{l}}
\nonumber\\
&=& (\rho^{2\b-2}g^{\z\bar{\z}}) (\rho^{2-2\b} u_{,\z\bar{\z}}) + \sum_{k>1}  g^{\bar{\z}k} u_{,k\bar{\z}}
\nonumber\\
&+& \sum_{l>1}  g^{\z\bar{l}}  u_{,\z\bar{l}}  + \sum_{k,l>1}  g^{\bar{l}k} u_{,k\bar{l}}.
\end{eqnarray}
Notice that we have  
$$ \rho^{2\b-2}g^{\z\bar{\z}},\  g^{\bar lk} \in\cC^{2,\a,\b},$$
by Lemma (\ref{lem-vol2}). Hence it is enough to prove that the mixed terms have the correct growth rate
\begin{equation}
\label{555}
g^{\bar{l}\z}u_{,\z\bar l},\ g^{\bar\z k}u_{,k\bar\z}\in \cC^{2,\a,\b}.
\end{equation}
By the expansion formula of the inverse matrix of the metric $\omega$, we have 
$$ g^{\bar l\z} = \frac{|\z|^{2-2\b}\det G_{\z\bar l}}{|\z|^{2-2\b}\det g}. $$
The denominator is in $\cC^{2,\a,\b}$, and the numerator can be written as a linear combination
$$ \det G_{\z\bar l} = \sum_{j>1} g_{j\bar\z} \hat{G}^j, $$ where $\hat{G}^j$ is some $\cC^{2,\a,\b}$ function for each $j>1$.
Therefore,
$$ g^{\bar l\z} u_{,\z\bar l} = \frac{\sum_{j>1} ( |\z|^{2-2\b} g_{j\bar\z} u_{,\z\bar l}) \hat{G}^j}{|\z|^{2-2\b}\det g}$$
is of class $\cC^{2,\a,\b}$ thanks to Lemma (\ref{lem-vol}).

According to Corollary (\ref{cor-ric}), the second term  $R^{\bar{\nu}\mu}\nabla_{\mu}\nabla_{\bar{\nu}} u$
in the definition of the operator $\cD$ is of class $\cC^{,\a,\b}$.
Therefore, the operator $\cD$ is well defined, and it is real simply from its definition. 
\end{proof}

\begin{lemma}
\label{lem-L2}
The operator $L_g|_{X-D}$ is $L^2$ acting on $\cC^{4,\a,\b}$ functions. Moreover, it satisfies the following integration by parts formula
\begin{equation}
\label{L2}
\int_{X-D} ( \cD u,  v) \omega^n = \int_{X-D} \langle L_g u, L_g v\rangle_g \omega^n,
\end{equation}
for any $u,v\in \cC^{4,\a,\b}$. 
\end{lemma}
\begin{proof}
It is enough to assume $u,v$ are real valued functions. The standard computation \cite{Fut} implied the following equation on $X-D$
$$ \cD u = \sum_{\a,\b,\mu,\nu =1}^n g^{\bar{\b}\a}g^{\bar{\nu}\mu} \nabla_{\a}\nabla_{\mu}\nabla_{\bar{\b}}\nabla_{\bar{\nu}}u. $$
However, we can not use this formula to do integration by parts (IbP), since some 3rd. order derivatives of the metric $g$ are involved in the RHS of above equation. 

Now consider the Laplacian square term first, and we claim the following 
\begin{eqnarray}
\label{L22}
\int_{X-D} (\Delta_{\omega}^2 u, v) \det g &= & \int_{X-D} g^{\bar{\b}\a}g^{\bar{\nu}\mu} (\nabla_{\a}\nabla_{\bar{\b}}\nabla_{\mu}\nabla_{\bar{\nu}}u)\bar{v}\det g
\nonumber\\
&= & - \int_{X-D} g^{\bar{\b}\a}g^{\bar{\nu}\mu} (\nabla_{\bar{\b}}\nabla_{\mu}\nabla_{\bar{\nu}}u) \nabla_{\a} \bar{v} \det g
\nonumber\\
& = &  - \int_{X-D} g^{\bar{\b}\a}g^{\bar{\nu}\mu} (\nabla_{\mu}\nabla_{\bar{\b}}\nabla_{\bar{\nu}}u - 
R^{\bar{\gamma}}_{\bar{\nu}\mu\bar{\b}} \nabla_{\bar{\gamma}} u ) \nabla_{\a} \bar{v} \det g.
\end{eqnarray}
Here we have to explain the IbP from the first line to the second line. First notice that each integrand involved in equation (\ref{L22}) is $L^1(\omega^n)$.  
Moreover, take $\chi_{\ep}$ to be a smooth cut off function \cite{LZ} supported outside the divisor such that 
$$ \d \chi_{\ep} = \frac{\ep \eta}{\z\log\rho} (d\z - \z \d\psi ), $$
where $\eta$ is a smooth function supported in an annual region outside $D$, and $\psi$ is another smooth function. 
We obtained the following IbP formula on $X$
\begin{eqnarray}
\label{IbP}
\int_{X} g^{\bar{\b}\a}g^{\bar{\nu}\mu} (\nabla_{\a}\nabla_{\bar{\b}}\nabla_{\mu}\nabla_{\bar{\nu}}u)\chi_{\ep}\bar{v}\det g &=&
  \int_{X} g^{\bar{\b}\a}g^{\bar{\nu}\mu} (\nabla_{\bar{\b}}\nabla_{\mu}\nabla_{\bar{\nu}}u) \nabla_{\a} (\chi_{\ep}\bar{v}) \det g
\nonumber\\
& = & \int_{X} g^{\bar{\b}\a}g^{\bar{\nu}\mu} (\nabla_{\bar{\b}}\nabla_{\mu}\nabla_{\bar{\nu}}u) \nabla_{\a} \chi_{\ep} \bar{v} \det g
\nonumber\\
&+& \int_{X} g^{\bar{\b}\a}g^{\bar{\nu}\mu} (\nabla_{\bar{\b}}\nabla_{\mu}\nabla_{\bar{\nu}}u) \nabla_{\a}\bar{v} \chi_{\ep} \det g
\end{eqnarray}
It is enough to prove that all the terms involving $\nabla\chi_{\ep}$ converges to zero, when $\ep$ does.
Then the claim follows from the following estimates
$$ g^{\bar{l} k}g^{\bar{j}i} u_{,\bar l i \bar j} \d_{k} \chi_{\ep} = \ep o(1), \  \ g^{\bar{\z} k}g^{\bar{j}i} u_{,\bar\z i \bar j} \d_{k} \chi_{\ep}= \ep o(1),\ \ 
 g^{\bar{l} k}g^{\bar{j}\z} u_{,\bar l \z \bar j} \d_{k} \chi_{\ep}= \ep o(1), $$
$$ g^{\bar{l} k}g^{\bar{j}i} u_{,\bar l i \bar j} \d_{k} \chi_{\ep} = \ep o(1),\ \ 
 g^{\bar{l} \z}g^{\bar{j}i} u_{,\bar l i \bar j} \d_{\z} \chi_{\ep}  = \ep o(\rho^{-\b}),\ \  g^{\bar{\z} k}g^{\bar{\z}i} u_{,\bar \z i \bar \z} \d_{k} \chi_{\ep} = \ep o(1),$$
$$ g^{\bar{\z} k}g^{\bar{j}\z} u_{,\bar \z \z \bar j} \d_{k} \chi_{\ep}= \ep o(1), \ \ 
 g^{\bar{l} k}g^{\bar{\z}\z} u_{,\bar l \z \bar \z} \d_{k} \chi_{\ep}= \ep o(1),\ \ 
  g^{\bar{\z} \z}g^{\bar{j}i} u_{,\bar \z i \bar j} \d_{\z} \chi_{\ep} = \ep o(\rho^{-\b}),$$
$$  g^{\bar{l} \z}g^{\bar{\z}i} u_{,\bar l i \bar \z} \d_{\z} \chi_{\ep} = \ep o(\rho^{-\b}),\ \ 
 g^{\bar{l} \z}g^{\bar{j}\z} u_{,\bar l \z \bar j} \d_{\z} \chi_{\ep}= \ep o(\rho^{-\b}),\ \ g^{\bar{\z}\z}g^{\bar{\z}\z} u_{,\bar\z\z\bar\z} \d_{\z} \chi_{\ep}  = \ep o(\rho^{-\b}).$$
\\

On the other hand, we have
\begin{eqnarray}
\label{L23}
\int_{X-D} Ric (\ddbar u, v) \det g &=& \int_{X-D} \nabla_{\mu}(R^{\bar{\a}\mu}\nabla_{\bar{\a}} u) \bar v \det g
\nonumber\\
&=& - \int_{X-D}  (R^{\bar{\a}\mu}\nabla_{\bar{\a}} u) \nabla_{\mu} \bar v \det g
\end{eqnarray}
Here we used the fact that $\nabla R =0$ on $X-D$, and the growth estimates 
$$ R^{\bar\z\z}\d_{\bar\z}u\d_{\z}\chi_{\ep} = \ep  o(\rho^{-\b}),\ \ \ R^{\bar l \z} \d_{\bar l} u \d_{\z}\chi_{\ep} = \ep o(\rho^{-\b}),  $$
$$ R^{\bar\z k} \d_{\bar\z}u \d_{k}\chi_{\ep} = \ep o(1),\ \ \  R^{\bar l k} \d_{\bar l} u \d_{k} \chi_{\ep} = \ep o(1). $$
Then equation (\ref{L23}) follows. Moreover, equation (\ref{L22}) implies that 
\begin{eqnarray} 
\label{L24}
\int_{X-D} (\cD u, v) \omega^n &=& -\int_{X-D} g^{\bar\b \a}g^{\bar\nu \mu} u_{,\bar\nu\bar\b \mu } v_{,\a}\det g
\nonumber\\
&=& \int_{X-D} g^{\bar\b \a}g^{\bar\nu \mu} u_{,\bar\nu\bar\b  } \bar{v}_{,\a\mu}\det g
\end{eqnarray}
The last equation follows from some similar growth estimates involved in equation (\ref{IbP}). 
\end{proof}

Next we are going to establish a one-one correspondence between the kernel of $\cD$ and  all holomorphic vector fields on $X$ tangential to the divisor. 
\begin{lemma}
\label{lem-vector}
Suppose $u\in \cC^{4,\a,\b}$ is in the kernel of $\cD$. Then the lifting of $u$ by the metric $\omega$ is a holomorphic vector field on $X$ tangential to $D$.
\end{lemma}
\begin{proof}
Let $Z$ be the lifting of $\dbar u$ under the metric $\omega$. That is to say 
$$Z = \uparrow^{\omega}\dbar u,$$
and 
$$Z^{\z} = g^{\bar\z\z}\d_{\bar\z}u + \sum_{l>1} g^{\bar l \z}\d_{\bar l} u, $$
$$Z^{k} = g^{\bar\z k} \d_{\bar\z}u + \sum_{l>1} g^{\bar l k} \d_{\bar l}u. $$
Notice that $Z^{\z}$ converges to zero and $Z^{k}$ converges to its smooth part $\sum_{l>1}g^{\bar l k} \d_{\bar l}u$, when the point approaches the divisor.
Moreover, the vector field $Z$ is holomorphic on $X-D$, since $L_g u =0$ outside the divisor thanks to Lemma (\ref{L2}). 
Therefore, near each point $p\in D$, we can extend the coefficients $Z^{\z}, Z^{k}$ across the divisor as holomorphic functions. 
Let $\tilde{Z}$ be the holomorphic extension of $Z$. In prior, the value of $\tilde{Z}$ depends on the local coordinate. 
However, the normal direction of $\tilde{Z}$ always vanishes, and we have 
$$ \tilde{Z}|_D = \sum_{k>1} (\sum_{l>1} g^{\bar l k} \d_{\bar l}u) \frac{\d}{\d z_k} = \uparrow^{\omega|_D} \dbar ( u|_D),$$
which implies our extension does not depend on the local coordinates. 
\end{proof}

\begin{proof}[Proof of Theorem \ref{thm-reductive}]
Let $Y$ be an element of $\mathfrak{h}(X;D)$. 
According to the hodge decomposition (Lemma \ref{lem-hodge}), the induced $(0,1)$ form $\tau$ of $Y$ can be written as 
$$\tau = \gamma + \dbar u, $$ where $\gamma$ is $\dbar$-harmonic ($\dbar \gamma =0$, $\vartheta\gamma =0$) outside of $D$. Then we have on $X-D$
\begin{eqnarray}
\label{new4-001}
\cD u &=& \nabla_{\a}\nabla_{\b} \nabla^{\b}\nabla^{\a} u = \nabla_{\a}\nabla_{\b} \nabla^{\b} ( Y^{\a} - \tilde{\gamma}^{\a}) 
\nonumber\\
&=& - \nabla_{\a}( \nabla^{\b}\nabla_{\b} \tilde{\gamma}^{\a} - R_{\b\ \ \mu}^{\ \b\a} \tilde{\gamma}^{\mu}) = -(\nabla^{\bar\mu}R) \gamma_{\bar\mu} 
\nonumber\\
&=& 0,
\end{eqnarray}
where $\tilde{\gamma}$ is the lifting of $\gamma$ by $\omega$. Notice that $\cD u$ is of class $\cC^{,\a,\b}$, and then the function $u\in \cC^{4,\a,\b}$ is in the kernel of $\cD$. 
Furthermore, we have $$L_g u =0,$$  thanks to Lemma (\ref{lem-L2}). 
And Lemma (\ref{lem-vector}) implies that the lifting $\uparrow^{\omega}\dbar u $ is an element of $\mathfrak{h}'(X;D)$.
Therefore, there exists a one-one correspondence between $\mathfrak{h}'(X;D)$ and the kernel of $\cD$ under the action $grad_g$,
and the reductivity follows since $\cD$ is a real operator. 

On the other hand, $\tilde{\gamma} = Y -  \uparrow^{\omega}\dbar u$ is in the space $\mathfrak{h}(X;D)$. 
Hence we have 
$$\nabla_{\bar\nu} \tilde{\gamma}^{\a} = \d_{\bar\nu} \tilde{\gamma}^{\a} = 0,$$ on $X$. 
Moreover, we have 
$$\nabla_{\mu}\tilde\gamma^{\a} = (\uparrow \d \gamma)^{\a}_{\mu} =0,$$ on $X-D$, 
since $\gamma$ is a conjugate holomorphic $(0,1)$ form on $X-D$. Therefore, $\tilde\gamma$ is in the space $\mathfrak{a}(X;D)$, and we have the decomposition 
$$Y = \tilde\gamma + \uparrow^{\omega} \dbar u. $$
Moreover, we claim the following communication relations
\begin{equation}
\label{comm1}
[\mathfrak{a}(X;D), \mathfrak{a}(X;D)] = \{ 0\},
\end{equation}
\begin{equation}
\label{comm2}
[\mathfrak{h}(X;D), \mathfrak{h}(X;D)] \subset \mathfrak{h}'(X;D).
\end{equation}
Equation (\ref{comm1}) is easy. For equation (\ref{comm2}), consider two elements $Y_1, Y_2 \in \mathfrak{h}(X;D)$ and their respective decomposition
$$ Y_1 = \tilde{\gamma}_1 + \uparrow^{\omega}\dbar u_1,\ \ \ Y_2 = \tilde{\gamma}_2 + \uparrow^{\omega}\dbar u_2.  $$
Then outside of the divisor $D$, the holomorphic vector field $Z: = [X, Y]$ tangential to $D$ can be expressed by 
$$ Z = \uparrow^{\omega}\dbar \varphi, $$ 
where the function 
$$ \varphi:= Y_1^{\mu} \d_{\mu} u_1 - Y_2^{\mu} \d_{\mu}u_1 $$
is of class $\cC^{,\a,\b}$. And it satisfies the following Laplacian equation 
$$ \vartheta(\downarrow Z) = \Delta_{\omega}\varphi, $$ on $X-D$.
Thanks to Proposition (\ref{prop-reg}) and (\ref{prop-sch}), we conclude that $\varphi\in\cC^{4,\a,\b}$ is in the kernel of $\cD$.  
Then our claim follows, and this complete the proof of Theorem (\ref{thm-reductive}).

\end{proof}

\subsection{Further remarks}
Based on the convexity of conic Mabuchi's functional \cite{Li}, the remaining task is to deform a conic cscK metric to twisted cscK metrics from the bifurcation technique \cite{CPZ}. 
In order to do this, a linear theory is established for the 4th. order elliptic operator $\cD$ in the conic setting \cite{LZ2}. 

Then it is natural to prove that the Lichnerowicz operator $\cD: \cC^{4,\a,\b} \rightarrow \cC^{,\a,\b}$ has a countable, discrete spectrum $\Lambda \subset \mathbb{R}^+\cup\{0\}$.
Moreover, it is expected that the first eigenvalue $\lambda_1$ of $\cD$ is strictly positive.  
And this is done by invoking the standard elliptic theorems in the conic setting \cite{CW}.  

On the other hand, it is interesting to ask the reducitivity when the cone angle is larger than $\pi$. 
There are several difficulties in this case.
First, even if the Ricci tensors are cone admissible, the growth rate of the Riemannian curvatures are not expected to be, since Brendle's trick fails here. 
Second, if the metric has no model growth in the 3rd. order, 
then the integration by parts formula for the Lichnerowicz operator $\cD$ is unlikely to be true,
Therefore, a new technique from Several Complex Variables will be utilized to investigate this case as we did in \cite{LZ}.

\section{Fundamentals of $\cC^{4,\a,\b}$ space}
\label{sec-fund}
We are going to give some fundamental facts about the space $\cC^{4,\a,\b}$ in this section. 
Recall that a $\cC^{4,\a,\b}$ function $u$ is defined by the equation 
\begin{equation}
\label{app-c4}
\Delta_{\Omega_0}u\in \cC^{2,\a,\b}
\end{equation}
for any $u\in\cC^{2,\a,\b}$. We claim that this relation is independent of holomorphic coordinate charts,
and then we can introduce a (complete) norm to make it as a H\"older space.

For simplicity of demonstration, we assume $n=2$ in this section. The general case follows by a similar argument.
Let $p$ be a point on the divisor $D$, and $U, V$ are two coordinate charts such that $p\in U\cap V$. 
Suppose $(\z, z)$  is a holomorphic coordinate system on $U$ such that the divisor is defined by $\{ \z = 0\}$ in $U$,
and $(\z', w)$ is a holomorphic coordinate system on $V$ such that $D\cap V: = \{\z' =0 \}$. 
Then we can write a biholomorphic map  
$$ F(\z', z) : = \{ \z(\z', w), z(\z', w) \} $$
such that $\z' = \z\cdot f(\z, z)$ for some everywhere non-zero holomorphic function $f$ near $p$. 
Similar, we have $\z = \z' \cdot g$, where $g = f^{-1}(\z',w)$. 

\subsection{Well defined}
The local model metrics in these two charts are defined by 
$$ \Omega_0 = \b^{2} |\z|^{2\b-2} i d\z \wedge d\bar\z + idz\wedge d\bar z $$ on $U$, and 
$$ \Omega'_0 = \b^2 |\z'|^{2\b-2} i d\z'\wedge d\bar\z' + i dw\wedge d\bar w. $$

\begin{prop}
\label{prop-inv}
The relations $\Delta_{\Omega_0} u\in\cC^{2,\a,\b}$ and $\Delta_{\Omega'_0} u\in \cC^{2,\a,\b}$ are equivalent. 
\end{prop}
\begin{proof}
Suppose $u\in\cC^{2,\a,\b}$ satisfies $\Delta_{\Omega_0}u \in \cC^{2,\a,\b}$. We can compute its derivatives in the $V$ coordinate by the chain rule 
$$ \d_{\z'} u = \frac{\d\z}{\d\z'} \d_{\z}u + \frac{\d z}{\d\z'} \d_z u, $$
$$ \d_{w} u = \frac{\d\z}{\d w} \d_{\z}u + \frac{\d z}{\d w} \d_w u. $$
Note that we have 
$$ \frac{\d\z}{\d\z'} = g + \z' \d_{\z'} g, \ \ \ \frac{\d\z}{\d w} = \z' \d_{w}g. $$
This implies that we have $\d\z / \d\z' \rightarrow g$ and $\d\z / \d w \rightarrow 0$ as the point approaching to the divisor. 
Next the complex Hessian can also be computed. 
\begin{eqnarray}
\label{ap-001}
\frac{\d^2 u}{\d\z\d\bar\z'} = \frac{\d\bar\z}{\d\bar{\z'}} \d_\z\d_{\bar\z} u + \frac{\d\bar z}{\d\bar\z'} \d_{\z}\d_{\bar z}u 
\end{eqnarray}
\begin{eqnarray}
\label{ap-002}
\frac{\d^2 u}{\d z\d\bar\z'} = \frac{\d\bar\z}{\d\bar{\z'}} \d_z \d_{\bar\z} u + \frac{\d\bar z}{\d\bar\z'} \d_{z}\d_{\bar z}u 
\end{eqnarray}
Moreover, we have 
\begin{eqnarray}
\label{ap-003}
\frac{\d}{\d\z'} \left( \frac{\d u}{\d\bar\z'} \right) &=& \frac{\d\z}{\d\z'} \d_\z\d_{\bar\z'} u + \frac{\d z}{\d\z'} \d_{\bar\z'}\d_{z}u 
\nonumber\\
&=& \left| \frac{\d\z}{\d\z'}\right|^2 \d_\z\d_{\bar\z}u + \left| \frac{\d z}{\d\z'}\right|^2 \d_z\d_{\bar z}u 
\nonumber\\
&+& 2 Re \left\{ \left( \frac{\d\z}{\d\z'}\right) \ol{ \left( \frac{\d z}{\d\z'} \right)} \d_\z\d_{\bar z} u\right\},
\end{eqnarray}
and 
\begin{eqnarray}
\label{ap-004}
\frac{\d}{\d w} \left( \frac{\d u}{\d\bar w} \right) &=& \frac{\d\z}{\d w} \d_\z\d_{\bar w} u + \frac{\d z}{\d w} \d_{\bar w}\d_{z}u 
\nonumber\\
&=& \left| \z' \d_w g \right|^2 \d_\z\d_{\bar\z}u + \left| \frac{\d z}{\d w}\right|^2 \d_z\d_{\bar z}u 
\nonumber\\
&+& 2 Re \left\{ \left( \z' \d_w g \right) \ol{ \left( \frac{\d z}{\d w} \right)} \d_\z\d_{\bar z} u\right\}.
\end{eqnarray}

It is enough to prove $|\z'|^{2-2\b}\d_{\z'} \d_{\bar\z'}u $ and $\d_w\d_{\bar w}u$ are of class $\cC^{2,\a,\b}$. 
Note that we have 
$$ |\z|^{2-2\b}\d_\z\d_{\bar\z} u\in\cC^{2,\a,\b},\ \ \  \d_z\d_{\bar z} u\in \cC^{2,\a,\b}$$
by the local definition of $\cC^{4,\a,\b}$. 
Hence the term $|\z'|^{2-2\b} \d_\z\d_{\bar\z}u$ is of class $\cC^{2,\a,\b}$ by the equivalence between $\z$ and $\z'$. 
Moreover, the term $|\z|^{2-2\b} \d_z\d_{\bar z} u$ is also of class $\cC^{2,\a,\b}$ since $ |\z|^{2-2\b}$ is thanks to the small angle condition. 
Then the Proposition follows from Lemma (\ref{lem-3partial}).
\end{proof}

\subsection{Functional spaces}
According to Donaldson \cite{Don12}, we are going to investigate the norms for the conic H\"older spaces. 
In fact, it is enough to consider these H\"older norms near the divisor $D$.

Let $p$ be a point on the divisor $D$, and $U$ be a holomorphic coordinate centered at $p$ such that the divisor $D$ is locally defined by 
$\{\z = 0 \}$. 

Put $B_1$ as the unit ball in $U$ centered at $p$. For any real valued function $u\in\cC^{,\a,\b}$, its H\"older norm on $B_1$ is introduced as 
$$ || u ||_{\cC^{,\a,\b} (B_1)} = \sup_{B_1} |u| + [\tilde{u}]_{C^\a(B_1)}, $$ 
where 
$$ \tilde{u} (|\z|^{\b-1}\z,z_2,\cdots,z_n): = u(\z,z_2,\cdots,z_n). $$
If $u$ is of class $\cC^{1,\a,\b}$ and compactly supported in $B_1$, then its H\"older norm is defined as 
$$ || u ||_{\cC^{1,\a,\b}(B_1) } =  | \tilde u |_{C^{1}(B_1)} +  [ D \tilde{u} ]_{C^{,\a}(B_1)}.$$
Note that locally this norm is equivalent to
$$ || \rho^{1-\b}\d_\z u ||_{\cC^{,\a,\b}(B_1)} + \sum_{k>1} || \d_k u ||_{\cC^{,\a,\b} (B_1)},$$
since we have 
\begin{eqnarray}
\label{app-0001}
|\z|^{1-\b} \d_{\bar{\z}} u (\z,z_2,\cdots,z_n) &=& \frac{\b +1}{2} \d_{\bar\xi}\tilde{u} (\xi,z_2,\cdots,z_n) 
\nonumber\\
&+& \frac{\b-1}{2} \frac{\z^2}{|\z|^2} \d_{\xi}\tilde u (\xi,z_2,\cdots,z_n). 
\end{eqnarray}
Moreover, locally the norm for the H\"older space $\cC^{2,\a,\b}$ is defined by
\begin{eqnarray} 
|| u ||_{\cC^{2,\a,\b}} &=& \sum_{k, i >1} || \d_k\d_i u ||_{\cC^{,\a,\b}}  + \sum_{k,l>1} || \d_k\d_{\bar l} u ||_{\cC^{,\a,\b}} + \sum_{k>1} || \rho^{1-\b}\d_k\d_\z u ||_{\cC^{,\a,\b}}
\nonumber\\
&+& \sum_{k>1} || \rho^{1-\b} \d_k\d_{\bar\z} u ||_{\cC^{,\a,\b}}  + || \rho^{2-2\b} \d_\z \d_{\bar\z} u ||_{\cC^{,\a,\b}} 
\end{eqnarray}
Note that we are lack of directions $\d_\z\d_\z u$ and $\d_{\bar\z}\d_{\bar\z} u$ in the norm of $\cC^{2,\a,\b}$. 
However, the space $\cC^{2,\a,\b}$ is still a Banach space equipped with this norm. 
For the convenience of the reader, we present a proof here. 
\begin{lemma}
\label{lem-2ba}
The space $\cC^{2,\a,\b}$ is a Banach space.
\end{lemma}
\begin{proof}
Suppose $\{u_i\}$ is a Cauchy sequence in the space $\cC^{2,\a,\b}$, and then the following sequence 
$$ f_i : = \Delta_{\Omega_0} u_i \in\cC^{,\a,\b} $$
forms a Cauchy sequence under the norm $||\cdot||_{\cC^{,\a,\b}}$ (we can always assume $f$ is compactly supported near a point on the divisor by multiplying a cut off function to $u$). 
Therefore, there exists a function $f_{\infty}\in\cC^{,\a,\b}$ such that we have the convergence 
\begin{equation}
\label{app-0002}
|| f_i - f_{\infty} ||_{\cC^{,\a,\b}}\rightarrow 0. 
\end{equation}
On the other hand, there exists a function $u_{\infty}\in\cC^{1,\a,\b}$ which is the limit of $u_i$ 
$$ || u_i - u_{\infty} ||_{\cC^{1, \a,\b}} \rightarrow 0. $$
It is enough to prove that $u_{\infty}$ is of class $\cC^{2,\a,\b}$. 
Let $\Sigma$ be a relative compact domain outside of the divisor $D$, and then all functions $u_i$ are in the class $C^{2,\a'}(\Sigma)$ for some $\a' > 0$, 
by interior estimates (Lemma 6.16, \cite{GT}).
Moreover, the sequence $u_i$ converges in $C^{2,\a'}(\Sigma)$ thanks to interior Schauder estimate.
Note that this limit is nothing but $u_{\infty}$ outside of the divisor. 
Hence we have 
\begin{equation}
\label{app-0003}
 \Delta_{\Omega_0} u_{\infty}  = f_{\infty}
\end{equation}
on $X-D$, and the result follows from Donaldson's estimates (see Proposition 2.1 and 2.2, \cite{Br}).  
\end{proof}

Now consider the space $\cC^{4,\a,\b}$, and we can introduce the following norms locally on each holomorphic coordinate 
\begin{equation}
\label{510}
|| u ||_{\cC^{4,\a,\b}}:= || u ||_{\cC^{2,\a,\b}} + || v ||_{\cC^{2,\a,\b}}  
\end{equation}
where $$ v: = \Delta_{\Omega_0} u. $$

In order to obtain a norm space globally, choose a locally finite covering and a partition of unity subordinated to it. 
And we can glue the local norms together by this partition of unity. 
In prior, the global norm $|| \cdot ||_{\cC^{4,\a,\b}}$ depends on the chosen holomorphic coordinate charts and partition of unity.
However, these norms are all equivalent. This is because all smooth functions are of class $\cC^{4,\a,\b}$ under small angle condition.
Then the computation in Proposition (\ref{prop-inv}) together with 
Donaldson's Schauder estimate (Proposition 2.2 \cite{Br}) and Brendle's trick (see Appendix) imply that there exists a constant $C$ such that 
$$  C^{-1} || \Delta_{\Omega_0'} u ||_{\cC^{2,\a,\b}} \leq || \Delta_{\Omega_0} u ||_{\cC^{2,\a,\b}} \leq C || \Delta_{\Omega_0'} u ||_{\cC^{2,\a,\b}}, $$
for $\Omega_0$ and $\Omega_0'$ on two different holomorphic coordinate charts.

\begin{lemma}
The space $\cC^{4,\a,\b}$ is a Banach space equipped with this norm $|| \cdot ||_{\cC^{4,\a,\b}}$. 
\end{lemma}
\begin{proof}
We can again argue it in a small open neighborhood $U$ of a point on the divisor. 
Suppose $\{u_i\}$ is a Cauchy sequence under the norm $||\cdot ||_{\cC^{4,\a,\b}}$,
and then $\{ v_i\}$ also forms a Cauchy sequence in the Banach space $\cC^{2,\a,\b}$. 
Therefore, there exists a function $v_{\infty}\in \cC^{2,\a,\b}$ such that it is the limit 
$$ || v_i - v_{\infty} ||_{\cC^{2,\a,\b}}\rightarrow 0. $$
On the other hand, there exist a limit $u_{\infty}\in\cC^{2,\a,\b}$ of $u_i$
$$ || u_i - u_{\infty} ||_{\cC^{2,\a,\b}} \rightarrow 0. $$
Then it is enough to prove $u_{\infty} \in \cC^{4,\a,\b}$. 
Now for any point $q$ which is close but NOT on the divisor, we claim that 
$\Delta_{\Omega_0} u_{\infty}$ is of class $ C^{2,\a'}$ and it satisfies 
$$ \Delta_{\Omega_0} u_i \rightarrow \Delta_{\Omega_0} u_{\infty}, $$ 
in $C^{2,\a'}$ near $q$ for some $\a' > 0$. 

Suppose the distance between $q\in U$ and the divisor $D$ is larger than $1$ but less than $2$. 
Then we can consider a holomorphic coordinate ball $B_{1}(q)$ centered at $q$,
and introduce a cut-off function $\chi$ which equals to $1$ on $B_{1/2}(q)$ but compactly supported in $B_1(q)$. 
Note that the value of the functions $u_i$ and $\Delta_{\Omega_0}u_i$ are unchanged in the ball $B_{1/2}(q)$ if we replace $u_i$ by $\chi u_i$.
Hence we can assume the functions $u_i$ and $v_i$ is compactly supported in $B_1(q)$ for each $i$. 
Moreover, their limits $u_{\infty}$ and $v_{\infty}$ are also compactly supported in this ball. 

Now the sequence $v_i$ converges in $C^{2,\a'}(B_{1} (q))$ to a limit $\underline{v} \in C^{2,\a'}(B_1(q))$,
and then we can solve the Dirichlet problem on $B_1(q)$ with trivial boundary condition as 
\begin{equation}
\label{app-dir}
\Delta_{\Omega_0} \underline{u} = \underline{v}.
\end{equation}  
But note that the limit $\underline{v}$ can be nothing but $v_{\infty}$,
and then $\underline{u}$ is exactly $u_{\infty}$ by the uniqueness of Dirichlet problem, and the claim follows. 

Therefore, we have the equation 
$$\Delta_{\Omega_0}^2 u_{\infty} = \Delta_{\Omega_0}v_{\infty}\in\cC^{,\a,\b}, $$
pointwise outside of the divisor. This implies that $\Delta_{\Omega_0} u_{\infty}$ is of class $\cC^{2,\a,\b}$, 
and the result follows from Donaldson's Schauder estimate. 

\end{proof}

\subsection{From local to global}

We will give an equivalent global definition of the space $\cC^{4,\a,\b}$ here. Recall that Donaldson metric $\Omega$ is a conic K\"ahler metric on $X$ with cone angle $2\pi\b$ along the divisor $D$.  
And it is quasi-isometric to $\Omega_0$ near any point $p$ on the divisor. Therefore, we can introduce the following definition. 
\begin{defn}
\[
\cS^{4,\a,\b}: = \{ u\in \cC^{2,\a,\b}\cap C^4_{\mathbb{R}}(X-D)|\ \Delta_{\Omega}u\in\cC^{2,\a,\b} \}. 
\]
\end{defn} 

Suppose $u$ is a function of the class $\cS^{4,\a,\b}$, and then $\Delta_{\Omega} u$ is of class $\cC^{2,\a,\b}$. 
Notice that the metric $\Omega$ itself is of class $\cC^{4,\a,\b}$. Hence $u$ is of class $\cC^{4,\a,\b}$ by the same argument in the proof of Proposition (\ref{prop-sch}). 
This implies $\cS^{4,\a,\b} \subset \cC^{4,\a,\b}$. 

On the other hand, suppose $u$ is a function of class $\cC^{4,\a,\b}$. Then it is easy to compute in the tangential direction 
$$ (\Delta_{\Omega}u)_{,j\bar j} \in \cC^{,\a,\b}.  $$

In the normal direction, we have the following computation
\begin{equation}
\label{new-005}
|\z|^{2-2\b} \d_{\z}\d_{\bar\z}(\Delta_{\Omega} u) = \Delta_{\Omega} (|\z|^{2-2\b} \d_{\z}\d_{\bar\z} u) + \emph{two and three tensors}. 
\end{equation}
This is the same calculation as in equation (\ref{005}). 
And notice that all two and three tensors are cone admissible since $u\in\cC^{4,\a,\b}$, and the curvature tensors of $\Omega$ are also cone admissible. 
Therefore, we have $\Delta_{\Omega}u\in\cC^{2,\a,\b}$. And we proved the following.

\begin{prop}
\label{prop-glo}
We have 
$$ \cS^{4,\a,\b} = \ \cC^{4,\a,\b}. $$
\end{prop}

Moreover, we can also introduce the following norm for any $u\in\cS^{4,\a,\b}$
$$ || u ||_{\cS^{4,\a,\b}} = || u ||_{\cC^{2,\a,\b}} + || \Delta_{\Omega} u ||_{\cC^{2,\a,\b}}. $$
This norm is equivalent to the previous defined norms $||\cdot ||_{\cC^{2,\a,\b}}$ thanks to the calculation in equation $(\ref{new-005})$. 
\section{Appendix}
\label{app}
For the convenience of the reader, we present details of Brendle's trick \cite{Br} in this section. 
Let $D: = B_1(0) $ be the holomorphic unit disk in $\mathbb{C}$, and $D^*$ be the punctured disk. 
Suppose $w$ is a uniformly bounded smooth function on $D^*$, and $h$ is a H\"older continuous function on $D$ such that 
$$h(z): = |z|^{2\beta -2} f(|z|^{\beta -1}z),$$
for some $C^{\a}$ continuous function $f$ defined on $D$ with $f(0) = 0$. Hence we have 
$$ | h(z) | \leq [\tilde f]_{C^\a} |z|^{2\beta-2 + \a\b},  $$ 
where $\a\b < 1 - 2\b$, and $\beta < 1/2$. Put $|z| = r$, and 
\begin{equation}
\label{app2-name}
 F(z): = \d_z \d_z w(z) + (1-\b)z^{-1} \d_z w(z) , 
\end{equation}
and then we claim the following:

\begin{prop}\label{brendle}
Suppose the function $w$ satisfies the following Laplacian equation on $D^*$: 
\begin{equation}
\label{lap}
\frac{\d^2 w}{ \d z \d \bar{z}}  = h.
\end{equation}
Then we have the following estimate near the origin of $D$:
$$ || \ r^{2-2\b} F(z)  \ ||_{\cC^{,\a,\b}} \leq  C || f ||_{\cC^{,\a,\b}}. $$
\end{prop}

We will demonstrate our proof from step to step. 

\subsection{Distribution theory}
We first claim equation (\ref{lap}) holds on $D$ in the sense of distributions. The argument is standard here. 
Notice that $\Delta w - h$ is a distribution supported at the origin. Hence we can write it as 
$$ \Delta w - h = \sum_{|\a | = 0} C_{\a} \d^{\a}\delta, $$
where $|\a|$ is the order of the index $\a$, and only finite many constants $C_{\a}$ are non-zero. 

Next we are going to prove by contradiction. Suppose there is a constant $C_{\a}$ which is not zero. 
Define the following finite set of indexes. 
$$\mathcal{A}: = \{ l\geq 0 \  |\ \ \ C_{\a}\neq 0\ \emph{for some}\  |\a| =l  \}. $$
Take $\chi$ to be a smooth testing function compactly supported on $D$, such that 
$$\sum_{|\a| =l }C_{\a}\d^{\a}\chi(0) \neq 0$$ for each $l \in \mathcal{A}$.
Let $\chi_k = \chi(kz)$, and consider the following actions.
$$ h (\phi_k) = \int_D h\chi_k = \int_{|z|< 1/k} h\chi(kz) \leq \int_{|z| < 1/k} |h|, $$
and it converges to zero since $h\in L^1$.
On the other hand, we have another action.
$$\Delta w (\chi_k) = \int_D w \Delta \chi_k = \int_D k^2 w \Delta\chi (kz) = \int_D w(z/k) \Delta \chi (z), $$
but 
$$ |\Delta w (\chi_k)| = \left|  \int_D w(z/k) \Delta \chi (z) \right| \leq \sup_{D^*}|w| \int_D |\Delta \chi| < +\infty. $$
However, their difference behaves like 
$$(\Delta w - h) (\chi_k) = \sum_{|\a| \in \mathcal{A}} C_{\alpha} k^{|\a|} \d^{\a}\chi(0) \thicksim k^m,$$
where the index $m$ is the maximum element in $\mathcal{A}$. Hence this action diverges to infinity when $k\rightarrow +\infty$,
which is a contradiction to our previous boundedness result. 

\subsection{interior estimate}
For simplicity, we view $w$ as a bounded function on the punctured disk with radius $2$. 
Pick up an arbitrary point $z_0$ in $D^*$, and let $r_0 = \ep|z_0|$, for some $\ep < 1/2$.  
Consider the ball $B_{r_0}(z_0)$ centered at $z_0$ with radius $r_0$, we can define the following function on the open ball $B_1(0)$. 
$$ v(z): = w (z_0 + r_0 z). $$
Thanks to the local Schauder estimates, we have 
\begin{equation}
\label{sch}
|| \nabla v ||_{C^0(B_{1/2})} \leq  || v ||_{C^{2,\a}(B_{1/2})} \leq C ( ||\Delta v ||_{C^{,\a} (B_1)} + || v ||_{C^0(B_1)}).
\end{equation} 
The change of variables after rescaling implies the following equations: 
\begin{equation}
\label{res1}
r_0^2 \Delta w (z_0 + r_0 z) = \Delta v (z),
\end{equation} 
and 
\begin{equation}
\label{res2}
 [ r_0^{2} \Delta w ]_{C^{,\a}(B_{r_0}(z_0))} =  r_0^{-\a} [ \Delta v ]_{C^{,\a}(B_1)}.
\end{equation}
On the other side, 
$$ || \nabla v ||_{C^0(B_{1/2})} =  r_0 || \nabla w ||_{C^0 (B_{r_0}(z_0))}. $$
Combining above inequalities, we have for $r_0$ small
\begin{eqnarray}
\label{res3}
| \nabla w | (z_0) &\leq& || \nabla w ||_{C^0 (B_{r_0}(z_0))} 
\nonumber\\
&\leq& C ( r_0^{1+\a} [ h ]_{C^{,\a}(B_{r_0}(z_0))} + r_0 \sup_{B_{r_0}(z_0)} |h| + r_0^{-1}|| w ||_{C^0(B_{2r_0}\backslash  \{0\})} )
\nonumber\\
&\leq& C' r_0^{-1}.
\end{eqnarray}
The last inequality follows from the the estimate 
$$r_0^{1+\a}|| h ||_{C^{,\a}(B_{r_0}(z_0))} = O(r_0^{2\b + \a\b  -1 }), $$
when $\ep$ is small. And it can be proved by the following lemma.
\begin{lemma}
\label{lem-1}
Let $r_0 = \ep|z_0|$. When $\ep$ is small enough, for all $z_1, z_2\in B_{r_0}(z_0)$, we have 
\begin{equation}
\label{holder1}
| |z_1|^{\beta-1}z_1 - |z_2|^{\beta-1}z_2 | \leq C(\ep) r_0^{\beta-1} |z_1 - z_2|
\end{equation}
\end{lemma}
\begin{proof}
Let $z_1 = r_1, z_2 = r_2 e^{i\theta}$, and assume $r_2 < r_1$. Suppose $\theta =0$ first. Then equation (\ref{holder1}) follows from the H\"older continuity of the function $r^{\beta}$.
\begin{equation}
\label{holder2}
\frac{| r_1^{\beta} - r_2^{\beta}|}{|r_1 - r_2|} \leq C |r_1 - r_2|^{\beta-1} \leq C(\ep) r_0^{\beta-1}. 
\end{equation}
Next we can assume $0 < \theta < \pi/4$. when $\ep$ is small, the following inequality holds.
$$ |r_1^{\beta} - r_2^{\beta} e^{i\theta}| \leq |r_1^{\beta} - r_2^{\beta}| + \theta r_2^{\beta}, $$
where $\theta r_2^{\beta}$ is a small arc length. Moreover, the arc length $\theta r_2$ is equivalent to the line segment $|z_1 - z_2|$. 
That is to say, there exists a uniform bounded function $K(\theta) = \frac{\theta} {2\sin (\theta/2)}$, such that 
$$ \theta r_2  = K | z_1 - z_2|,$$ and here we can assume $ 1< K < 2$ for $\ep$ small. 
Finally, we have 
\begin{eqnarray}
\label{holder3}
\frac{| r_1^{\beta} - r_2^{\beta} e^{i\theta}|}{|r_1 - r_2 e^{i\theta}|}&\leq& \frac{| r_1^{\beta} - r_2^{\beta}|}{|r_1 - r_2 e^{i\theta}|} + \frac{\theta r_2}{|r_1 - r_2 e^{i\theta}|} r_2^{\b-1}
\nonumber\\
&\leq& \frac{| r_1^{\beta} - r_2^{\beta}|}{|r_1 - r_2|} + K r_2^{\b-1}
\nonumber\\
&\leq& C_1(\ep) r_0^{\beta -1} + Kr_2^{\beta-1} \leq C(\ep)r_0^{\beta -1},
\end{eqnarray}
where we used the fact that 
$$ |r_1 - r_2 e^{i\theta}| > |r_1 - r_2|, $$ 
for $r_1 > r_2$ and $\theta < \pi /4$.  
\end{proof}

In fact, these two distance functions are equivalent to each other in the ball $B_{r_0}(z_0)$
\begin{lemma}
\label{lem-distance}
For any $\ep$ small enough, and $z_1, z_2 \in B_{r_0}(z_0)$, there exists a constant $c(\ep)$ such that 
\begin{equation}
\label{holder12}
| |z_1|^{\beta-1}z_1 - |z_2|^{\beta-1}z_2 |  \geq c(\ep) r_0^{\beta-1} |z_1 - z_2|
\end{equation}
\end{lemma}
\begin{proof}
Again we can assume $r_2 < r_1$ and $z_0 = z_1$. When $\theta =0$, the inequality follows from the mean value equation 
$$ r_1^{\b} - r_2^{\b} = \b \tilde{r}^{\b-1} (r_1 -r_2) \geq C r_0^{\b-1} | r_1 - r_2| ,$$
for some $r_2 \leq \tilde r\leq r_1$. Note that we have the following estimates for $r_1 > r_2$
\begin{equation}
\label{app2-0110}
| r_1^\b - r_2^\b e^{i\theta} | \geq \max \{ | r^{\b}_1 - r^{\b}_2 |, K(\theta) \theta r_2^{\b} \}.
\end{equation}
Moreover, we have 
$$ |r_1 - r_2 e^{i\theta}| \leq 2 \max \{ |r_1 - r_2|, \theta r_2 \}. $$
Hence we can discuss case by case. First if $ | r_1 - r_2| \geq \theta r_2 $, then we have 
$$ \frac{ | r_1^{\b} - r_2^{\b} e^{i\theta}|}{ | r_1 - r_2 e^{i\theta}|} \geq \frac{ |r_1^{\b} - r_2^{\b}|}{ 2|r_1 -r_2|} \geq Cr_0^{\b-1}. $$
Second, if $|r_1 -r_2| \leq \theta r_2$, then we still have 
$$\frac{ | r_1^{\b} - r_2^{\b} e^{i\theta}|}{ | r_1 - r_2 e^{i\theta}|} \geq \frac{K\theta r_2^{\b}}{\theta r_2} \geq C r_0^{\b-1},  $$
and the result follows. 
\end{proof}

\begin{lemma}
We have 
\label{lem-2}
$$ [ h ]_{C^{\a}(B_{r_0}(z_0))} \leq C[\tilde f]_{C^\a} r_0^{2\beta-2 + \a(\b-1)},$$
for any $r_0$ small enough. 
\end{lemma}
\begin{proof}
For any pair $z_1,z_2\in B_{r_0}(z_0)$
\begin{eqnarray}
|h(z_1) - h(z_2)| &\leq& r_1^{2\b -2} | f(r_1^{\b-1}z_1 ) - f(r_2^{\b-1}z_2) | 
\nonumber\\
&+& (r_1^{2\b -2} -r_2^{2\b-2}) |f(r_2^{\b-1}z_2)|
\nonumber\\
&\leq& [\tilde f]_{C^{\a}}  r_1^{2\b -2} | r_1^{\b-1}z_1  - r_2^{\b-1}z_2 |^{\a}   + C'(r_1^{2\b -2} -r_2^{2\b-2}) r_2^{\a\b}.
\end{eqnarray}
Thanks to lemma (\ref{lem-1}), we have 
$$ |h(z_1) - h(z_2)|  \leq C'' [\tilde f]_{C^\a}  r_0^{2\b-2 + \a(\b-1)} |z_1 -z_2|^{\a} + C''' r_0^{2\b -3 + \a\b} |z_1 - z_2|. $$
And the result follows for $r_0$ small enough.
\end{proof}

\subsection{$W^{1,2}$ estimate and H\"older continuity }
Let's compute the $L^2$ norm of $\nabla w$ as follows. 
\begin{equation}
\int_{B_1 - B_{\ep}} |\nabla w|^2 =  \int_{B_1 - B_{\ep}} w h + \int_{\d B_1} w \langle \nabla w, \hat{r}\rangle - \int_{\d B_{\ep}}w \langle \nabla w, \hat{r}\rangle.
\end{equation}
The first two terms in above equation are uniformly bounded. 
And the third term can be estimated by 
$$\int_{\d B_{\ep}}w \langle \nabla w, \hat{r}\rangle \leq C \int_{\d B_{\ep}} |\nabla w| \leq \int_{|z|=\ep} r^{-1} < K,$$
for some uniform constant $K$. This implies that $\nabla w$ is in $L^2(D^*)$.
Then we can say there exists an $L^2(D)$ function $g$ (vector valued) such that $\nabla w = g $ on $D$.
In other words, $w$ is the weak solution of the Laplacian equation.
Then the $L^p$ regularity estimate implies $w\in W^{2,p}$ for some $1< p <2$  ($w\in W^{1,1}$ is enough!).
And the Sobolev embedding implies that $w\in C^{\a'}$ around the origin.  

\subsection{Cauchy's integral formula}
Let's apply the Cauchy integral formula to the punctured ball $\Omega = B_{1}\backslash B_{\ep}$. 
\begin{eqnarray}
\label{green}
2\pi i \frac{\d w}{\d z}(z_0) &=& \int_{\Omega} \frac{h(z)}{z-z_0} dz\wedge d\bar{z} + \int_{\d B_1} \frac{\d w / \d z}{z-z_0}dz 
\nonumber\\
&-& \int_{\d B_{\ep}} \frac{\d w / \d z}{z-z_0}dz.
\end{eqnarray}
The second term is uniformly bounded. And we claim the third term converges to zero as $\ep$ does. First, we can assume $w(0) = 0$ after adjusting a constant to it. 
Then $|w(z)|\leq C |z|^{\a'}$ for some small $\a'$.
Now inequality (\ref{res3}) actually implies the follows. 
$$ | \nabla w |(z_0) \leq C r_0^{\gamma -1}, $$
where $\gamma = \min \{ \a', 2\beta + \a\b\}$. 
$$ \left |  \int_{\d B_{\ep}} \frac{\d w / \d z}{z-z_0}dz \right | \leq |z_0|^{-1} \int_{\d B_{\ep}} r^{\gamma -1},$$
which converges to zero. Therefore, the value of $\d w / \d\z (\z_0)$ is completely determined its Newtonian potential
Therefore, we have for all $z_0\in B_{1/4}$ 
\begin{eqnarray}
\label{app2-001}
\left| \d_{z} w (z_0) \right| &\leq& \left| \int_{\Omega} \frac{h(z)}{z - z_0} dz\wedge d\bar{z}  \right| + C
\nonumber\\
&\leq&  \int_{B_1} |z - z_0|^{-1} |h(z)| dz \wedge d\bar{z} + C
\nonumber\\
&\leq& [\tilde f]_{C^\a} \int_{B_1} |z - z_0|^{-1} |z|^{2\b-2+\a\b} dz\wedge d\bar z + C
\nonumber\\
&\leq& C[\tilde f]_{C^{\a}} |z_0|^{2\b -1 + \a\b},
\end{eqnarray}
where we used the condition $ 0 < \a\b < 1- 2\b$ on the last line. Consequently, 
\begin{equation}
\label{app2-002}
|z|^{1-\b} | \d_z w | \leq C [\tilde f]_{C^{\a}} |z|^{\a\b},
\end{equation}
for all $z\in B_{1/4}$. 
 
Now put $v(z) = w(z_0 + r_0 z)$ again, and use the Schauder estimate for $\dbar$ operator (acting on functions) 
\begin{equation}
\label{app2-interior2}
|\d_z \d_z v| (0) \leq  ||  \d_z v ||_{C^{1,\a} (B_{1/2})} \leq C ( || \Delta v ||_{C^{,\a} (B_1)} + || v ||_{C^0(B_1)}).
\end{equation}
This implies 
\begin{eqnarray}
\label{app2-003}
|\d_z \d_z w |(z_0) &\leq& C |z_0|^{\a} [ h ]_{C^\a(B_{r_0}(z_0))}  + C' \sup_{B_{r_0}(z_0)} |h|+ C''|z_0|^{-1} \sup_{B_{r_0}(z_0)} |\d_z w| 
\nonumber\\
&\leq& C [\tilde f]_{C^\a} |z_0|^{2\b -2 + \a\b}. 
\end{eqnarray}

Combine equation (\ref{app2-002}) and (\ref{app2-003}), we reproduce Brendle's result 
\begin{equation}
\label{app2-brendle}
 |z|^{2-2\b} (\d_z \d_z w + (1-\b)z^{-1} w )  \leq  C [\tilde f] _{C^{,\a}} |z|^{\a\b}.
\end{equation}

\subsection{Outside the origin}
Take $\xi: = |z|^{\b-1}z$, and the H\"older distance function in the so called $w$-coordinate is denoted by 
$$ d (z_0, z_1):  = \frac{ | F(\xi_0 ) - F(\xi_1)|}{ | \xi_0 - \xi_1|^{\a}}. $$
We have seen the estimate for $d(z_0, 0)$ from equation (\ref{app2-brendle}), 
and then it is enough to consider the case when $z_0$ and $z_1$ are comparable, in the sense that $z_1$ is always in the ball $B_{r_0}(z_0)$. 
( If $z_0$ and $z_1$ are not comparable, then the estimate follows from the continuity of the H\"older distance function and triangle inequality. )

Put $I: =  || f ||_{\cC^{,\a,\b}}$, and we have the Cauchy integral formula again 
\begin{eqnarray}
\label{app2-green}
| \d_z w(z_0) - \d_z w(z_1)| &\leq& CI  |z_0 - z_1| \left|  \int_{B_1} \frac{|z|^{2\b -2 + \a\b}} { |z-z_0| |z - z_1|}dz \wedge d\bar z \right| + C
\nonumber\\
&\leq & CI r_0^{2\b - 2 + \a\b} |z_0 - z_1|. 
\end{eqnarray}
for all $z_1 \in B_{r_0}(z_0)$ and $z_0 \in B_{1/2}$. This implies the following 
\begin{eqnarray}
\label{app2-onederi}
| |z_1|^{1-2\b} \d_z w(z_1) - |z_0|^{1-2\b} \d_z w(z_0) | &\leq &  |z_0|^{1-2\b} | \d_z w(z_0) - \d_z w(z_1)  | 
\nonumber\\
& + & |\d_z w(z_1)| | |z_1|^{1-2\b} - |z_0|^{1-2\b}|
\nonumber\\
&\leq& C'I r_0^{-1+ \a\b} |z_0 - z_1| + C'' I | | \xi_1|^{\frac{1}{\b}-2} - |\xi_0|^{\frac{1}{\b}-2}|
\nonumber\\
&\leq& C''' I  |\xi_0 - \xi_1|^{\a}.  
\end{eqnarray}
The last inequality follows from Lemma (\ref{lem-distance}) and the fact 
$$r_0^{\a -1} |z_0 - z_1|^{1-\a} < 1$$ for all $z_1\in B_{r_0}(z_0)$. 

The interior estimate (equation \ref{app2-interior2}) also implies the following 
\begin{eqnarray}
\label{app2-secderi}
  r_0^{\a}[ \nabla \d_z w ]_{C^{\a}(B_{r_0}(z_0))} 
 &\leq& C r_0^{\a}[ h ]_{C^{\a}(B_{2r_0}(z_0))} +
 \sup_{B_{2r_0}(z_0)} |h|+  C' r_0^{-1} \sup_{B_{2r_0}} | \d_z w |
\nonumber\\
&\leq&  C'' I r_0^{2\b -2 + \a\b - \a}. 
\end{eqnarray}
Hence we can estimate it as 
\begin{eqnarray}
\label{app2-secderi2}
\frac{ |\d^2_z w(z_1) - \d^2_z w(z_0) |}{ | \xi_1 - \xi_0|^{\a}} &\leq& r_0^{\a-\a\b} [ \d^2_z w(\z) ]_{C^{\a}(B_{r_0} (z_0))}
\nonumber\\
&\leq& C I r_0^{2\b -2}, 
\end{eqnarray}
where we used the inequality 
$$ |z_0 - z_1| \leq C r_0^{1-\b}|\xi_0 - \xi_1|.  $$
Moreover, we have 
\begin{eqnarray}
\label{app2-final}
| |z_0|^{2-2\b} \d^2_z w(z_0) - |z_1|^{2-2\b} \d^2_z w(z_1) | &\leq& |z_0|^{2-2\b} |\d^2_z w(z_0) - \d^2_z w(z_1) |
\nonumber\\
&+& | \d^2_z w(z_1)| | |z_0|^{2-2\b} - |z_1|^{2-2\b} |. 
\nonumber\\
&\leq & CI |\xi_1 - \xi_0|^{\a}. 
\end{eqnarray}
The last inequality follows from the mean value equation 
$$  x^{2-2\b} - y^{2-2\b} = (2-2\b) \tilde{x}^{1-2\b} (x-y),$$
for some $\tilde x \in [y, x]$. 

\begin{proof}[Proof of Proposition (\ref{brendle})]
Combine inequality (\ref{app2-brendle}) and (\ref{app2-final}), and the result follows. 
\end{proof}

\end{document}